\documentclass[a4paper, 11pt]{article}
\usepackage{amsmath}
\usepackage{amsfonts}
\usepackage{amssymb}
\usepackage[english]{babel}
\usepackage{graphicx}
\usepackage{amsthm}
\usepackage[title]{appendix}
\usepackage{color}
\usepackage{soul}

\newtheorem{thm}{Theorem}[section]

\newtheorem{lem}[thm]{Lemma}
\newtheorem{prop}[thm]{Proposition}
\newtheorem{prob}[thm]{Problem}

\theoremstyle{definition}
\newtheorem{defn}[thm]{Definition}
\newtheorem{ex}[thm]{Example}

\numberwithin{equation}{section}

\newcommand{\FF}{\mathcal{F}}

\newcommand{\PP}{\mathcal{P}}
\newcommand{\LL}{\mathcal{L}}

\newcommand{\Z}{\mathbb{Z}}

\newcommand{\ol}{\overline}

\begin{document}
\title{Infinitely many cyclic solutions to the Hamilton-Waterloo problem with odd length cycles
\thanks{Work performed under the auspicies of the G.N.S.A.G.A. of the C.N.R. (National Research Council) of Italy and supported by M.I.U.R. project "Disegni combinatorici, grafi e loro applicazioni, PRIN 2008". The second author is supported by a fellowship of INdAM.}}

\author{Francesca Merola\thanks{Dipartimento di Matematica e Fisica,
Universit\`a Roma Tre, Largo Murialdo 1, 00146 Rome, Italy, 
email: merola@mat.uniroma3.it} \quad
Tommaso Traetta \thanks{Dipartimento di Matematica e Informatica, 
Universit\`a di Perugia, via Vanvitelli 1, 06123 Perugia, Italy, 
email: traetta@dmi.unipg.it, traetta.tommaso@gmail.com}}
\date{}

\maketitle
\begin{abstract}
\noindent
It is conjectured that for every pair $(\ell,m)$ of odd integers greater than 2 with
$m \equiv 1\; \pmod{\ell}$, there exists a cyclic two-factorization of $K_{\ell m}$ 
having exactly $(m-1)/2$ factors of type $\ell^m$ and all the others of type $m^{\ell}$. 
The authors prove the conjecture in the affirmative when 
$\ell \equiv 1\; \pmod{4}$ and $m \geq \ell^2 -\ell  + 1$.
\end{abstract}


\section{Introduction}

A {\it $2$-factorization} of order $v$ is a set $\FF$ of spanning $2$--regular subgraphs of $K_v$ (the {\it complete graph} of order $v$)
whose edges partition the edge set of $K_v$ or $K_v - I$ 
(the complete graph minus a $1$-factor $I$) 
according to whether $v$ is odd or even.
We refer the reader to \cite{We96} for the standard terminology and notation of elementary graph theory.

Note that every spanning $2$-regular subgraph $F$ of $K_v$ determines a partition 
$\pi=[\ell_1^{n_1}, \ell_2^{n_2},\ldots, \ell_t^{n_t}]$  of the integer $v$ where 
$\ell_1, \ell_2,\ldots,\ell_t$ are the distinct lengths of the cycles of $F$ and $n_i$ is the number of cycles in $F$ of length $\ell_i$ (briefly, $\ell_i$-cycles). 
We will refer to $F$ as a {\it $2$--factor of $K_v$} of {\it type $\pi$}. Of course, $2$--factors of the same type are pairwise isomorphic, and viceversa.

In this paper we deal with the well-known Hamilton-Waterloo problem (in short, HWP) which can be formulated as follows: 
given two non-isomorphic $2$-factors $F, F'$ of $K_v$, and two positive integers $r, r'$ summing up to $\lfloor (v-1)/2\rfloor$, then HWP asks for a $2$-factorization of order $v$ consisting of $r$ copies of $F$ and $r'$ copies of $F'$.
Denoted by $\pi$ and $\pi'$ the types of $F$ and $F'$, respectively, 
this problem will be denoted by $HWP(v;\pi, \pi'; r, r')$.

The case where $F$ and $F'$ are $2$-factors of the same type $\pi$ is known as the Oberwolfach problem,  $OP(v; \pi)$. 
This has been formulated much earlier by Ringel in 1967 for $v$ odd, while the case $v$ even was later considered in \cite{HuKoRo79}.  
Apart from OP$(6;[3^2])$, OP$(9;[4,5])$, OP$(11;[3^2,5])$,OP$(12,[3^4])$, none of which has a solution, the problem is conjectured to be always solvable: evidence supporting this conjecture can be found in \cite{BrRo07}. We only mention some of the most important results on the Oberwolfach problem recently achieved: OP($v;\pi$) is solvable for an infinite set of primes $v \equiv 1 \pmod{96}$ \cite{BrSch09}, when $\pi$ has exactly two terms \cite{Tr13}, 
and when every term of  $\pi$ is even \cite{BrDa11, Ha85}.
However, although the literature is rich in solutions of many infinite classes of  OP($v;\pi$), these only solve a small fraction of the problem which still remains open. 

As one would expect, there is much less literature on the Hamilton-Waterloo problem. 
Apart from some non-solvable instances of small order \cite{BrRo07}, the Hamilton-Waterloo problem HWP($v;\pi,\pi';r,r'$) is known to have a solution when $v$ is odd   and $\leq 17$ \cite{AdBr06, FrHoRo04, FrRo00}, and when $v$ is even and $\leq 10$ \cite{AdBr06, An07}.
When all terms of $\pi$ and $\pi'$ are even, with $r,r' >1$, a complete solution has been given in \cite{BrDa11}. Surprisingly, very little is known when $\pi$ or $\pi'$ contain odd terms, even when all terms of $\pi$ and all terms of $\pi'$ coincide. 
For example, HWP($v;[4^{v/4}],[\ell^{v/\ell}];r,r'$) has been dealt with in \cite{KeOz13, OdOz16} and completely solved only
when $\ell=3$ in \cite{BoBu, DaQuSt09, WangChenCao}, while HWP($v;[3^{v/3}],[v]; \lfloor v/2 \rfloor-1$, 1) is still open (see, \cite{DiLi091, DiLi092, HoNeRo04}). Other results can be found in \cite{AdBiBrEl02, BuRi05}. In this paper we make significant headway with the most challenging case of the Hamilton-Waterloo problem, that is, the one in which the two partitions contain only odd terms.
Further progress on this case has been recently made in \cite{BuDaTr}.

An effective method to determine a $2$-factorization $\FF$ solving a given Oberwolfach or Hamilton-Waterloo problem is to require that $\FF$ has 
a suitable {\it automorphism group} $G$, that is, a group of permutation on the vertices which leaves $\FF$ invariant.
One usually requires that $G$ fixes $k$ vertices and has $r\geq 1$ regular orbits on the remaining vertices.
If $k=1$, then $\FF$ is said to be {\it $r$-rotational} (under $G$) \cite{BuRi08}; if $r=1$ and $k\geq 1$, then $\FF$ is $k$-pyramidal \cite{BoMaRi09} -- note that $1$-pyramidal means $1$-rotational. Finally, $\FF$ is {\it sharply transitive} or {\it regular} 
(under $G$) if $(k,r)=(0,1)$.

The $1$-rotational approach and the $2$-pyramidal one have proved to be successful \cite{BuRi08, BuTr13, Tr13} in solving infinitely many cases of the Oberwolfach problem. 
On the other hand, all solutions of small  order given in \cite{DeFrHuMeRo10} turn out to be $r$-rotational for some suitable small $r$.

A $2$-factorization $\FF$ of order $v$ is regular under $G$ if we can label the vertices with the elements of $G$ so that for any $2$--factor $F\in \FF$ we have that $F+g$ is also in $\FF$. One usually speaks of a {\it cyclic}
$2$-factorization when $G$ is the cyclic group. 
The few known facts on regular solutions to the Oberwolfach problem only concern 
cyclic groups \cite{BuDel04, BuRaZu05, JoMo08}, 
elementary abelian groups and  Frobenius groups 
\cite{BuRi05}.
Concerning the Hamilton-Waterloo problem,  there are only two known infinite classes of solutions having a cyclic automorphism group; specifically, there exists a cyclic solution to
HWP$(18n+3;[3^{6n+1}],[(6 n+1)^3];3n,6n+1)$ \cite{BuRi05} and 
HWP$(50n+5;[5^{10n+1}],[(10n+1)^5];5n,20n+2)$ \cite{BuDa} for any $n\geq1$. 
These are instances of the following much more general problem.

\begin{prob}\label{prob} Given $\ell\geq3$ odd and given $n>0$, establish if there exists a cyclic solution to
HWP$(\ell(2\ell n+1);[\ell^{2\ell n+1}],[(2\ell n+1)^\ell];\ell n,
{\frac{(\ell-1)(2\ell n+1)}{2}})$.
\end{prob}

The above mentioned solutions settle the problem for $\ell=3$ and $\ell=5$.

In this paper, we build on the techniques of \cite{BuDa}
and consider the case $\ell=4k+1$; we manage to solve the problem for \textit{all} $\ell$ and \textit{all} $n\ge (\ell -1)/2$ by giving always ``cyclic'' solutions.
Our main result is the following.
\begin{thm}\label{mainth}
If $\ell \equiv 1\pmod{4}$, then Problem \ref{prob} admits a solution provided that $n\geq (\ell-1)/2$.
\end{thm}

We believe that techniques similar to those developed in this paper may 
be used to settle the case $\ell\equiv 3\;\pmod{4}$. However, in view of the complexity of this article, we postpone
the study of this case to a future work.

We finally point out that regular
solutions under non-cyclic groups can be found in \cite{BuRi05}: for example, 
given a positive integer $j$ and two odd primes $\ell, m$ such that $m^j \equiv 1\;\pmod{\ell}$,
there exists a regular solution to HWP$(\ell m^j; [\ell^{(m^j)}], [m^{\ell m^{j-1}}];$ 
$(\ell-1)m^j/2, (m^j-1)/2)$ 
under a Frobenius group of order $\ell m^j$.

\section{Some preliminaries}
The techniques used in this paper are based on those in the papers by Buratti and Rinaldi \cite{BuRi05} and Buratti and Danziger \cite{BuDa}; we collect here some preliminary notation and definitions, more details can be found in these references. 

In what follows, we shall label the vertices  of the complete graph $K_v$ with the elements of $\Z_v$;
if $\Gamma$ is a subgraph of $K_v$, we define the list of differences of $\Gamma$ to be the
multiset $\Delta\Gamma$ of all possible differences $\pm(a-b)$ between a pair
$(a,b)$ of adjacent vertices of $\Gamma$. More generally, the list of differences of a collection
$\{\Gamma_1,\dots,\Gamma_t\}$ of subgraphs of $\Gamma$ is the multiset union of
the lists of differences of all the $\Gamma_i$s.

A cycle $C$ of length $d$ in  $K_v$ is called {\it transversal} if  $d|v$ and the vertices of $C$ form a complete system of representatives for the residue classes modulo $d$, namely a transversal for the cosets of the subgroup of $\Z_v$ of index $d$.

A fundamental result we shall use in our construction is the following theorem, 
a consequence of a more general result proved in \cite{BuRi05} (see also \cite{BuDa}).

\begin{thm}\label{BR}
There exists a cyclic solution to HWP$(\ell m;[\ell^m],[m^\ell];r,r')$ if and only if the following conditions hold:
\begin{itemize}
\item $(r,r')=(\ell x,my)$ for a suitable pair $(x,y)$ of positive integers such that $2\ell x+2my=\ell m-1$;
\item there exist $x$ transversal $\ell$-cycles $\{A_1,\dots,A_x\}$ of $K_{\ell m}$ and $y$ transversal $m$-cycles
$\{B_1,\dots,B_y\}$ of $K_{\ell m}$ whose lists of differences cover, altogether, $\Z_{\ell m}\setminus\{0\}$ exactly once.
\end{itemize}
 \end{thm}

The reason for looking for a cyclic solution to the HWP with the particular parameter set of Problem \ref{prob} stems from the fact that for $\ell$ an odd integer and $m=2\ell n+1$ with $n$ a positive integer, taking $x=n$ and 
$y={\frac{\ell-1}{2}}$ will give $2\ell x+2my=\ell m-1$; thus Theorem \ref{BR} ensures that the problem will be solved if we construct $n$ transversal $\ell$-cycles of $K_{\ell(2\ell n+1)}$ and $(\ell -1)/2$ transversal $(2\ell n+1)$-cycles
of $K_{\ell(2\ell n+1)}$ whose list of differences cover $\Z_{\ell(2\ell n+1)}\setminus\{0\}$ exactly once.

For the sake of brevity, throughout the paper we will refer to an $\ell$--cycle or a $(2\ell n +1)$--cycle as a {\it short} cycle or a {\it long} cycle, respectively. Also, using the Chinese remainder theorem, we will identify $\Z_{\ell(2\ell n +1)}$ with $\Z_{2\ell n +1}\times \Z_{\ell}$. We point out to the reader that, after such an identification, a long (short) cycle $C$ of $K_{\ell(2\ell n +1)}$ is transversal if and only if the first (second) components of the vertices of $C$ are all distinct.

By $(c_0, c_1, \ldots, c_{\ell})$ we will denote both the path of length $\ell$ with edges 
$(c_0, c_1),$ $(c_1, c_2), \ldots, (c_{\ell-1}, c_{\ell})$ 
and the cycle that we get from it by joining $c_0$ and $c_{\ell}$. To avoid misunderstandings, we will always specify whether we are dealing with a path or a cycle. Finally, given two integers $a \leq b$, we will denote by $[a,b]$ the {\it interval} containing the integers $a, a+1, a+2, \ldots, b$. Of course, if $a<b$, then $[b,a]$ will be the empty set.

\section{The general construction}

In this section we shall sketch the steps needed to prove 
the main result of this paper,  Theorem \ref{mainth}. Let us start with some notation.

From now on $\ell$ will be an odd positive integers with $\ell \equiv 1 \pmod{4}$.
We will usually write $\ell=4k+1$: {we shall assume that $k>1$, since the case $k=1$, that is $\ell=5$, has been covered in \cite{BuDa}}. For $n$ a positive integer,
we denote by $\Z_{2\ell n+1}^*$ and $\Z_\ell ^*$ the set of nonzero-elements of $\Z_{2\ell n+1}$ and $\Z_\ell$, respectively. Also, $\Z_{2\ell n+1}^-$ will denote the set $\Z_{2\ell n+1}^* \setminus \{\pm 1, \pm \ell n\}$.

We point out to the reader that given a subset $S$ of $\Z_m$ $(m>1)$ and a subset $T$ of $\Z$, we write $S=T$ whenever
$T$ is a complete set of representatives for the residue classes $\pmod{m}$ in $S$. 
For example, $\Z_{m}=[0,m-1]$ or $\Z_{m}=[-\frac{m-1}{2}, \frac{m-1}{2}]$ when $m$ is odd.

To explain the construction used in what follows we first need to define a specific set $D$ of $2k-2$ positive integers which will play a crucial role in building the cycles we need.
\begin{defn}\label{setD}
For $\ell=4k+1$ and for {an integer $n>0$}, 
we denote by $D=\{d_{i1}, d_{i2} | i=1,\ldots,k-1\}$ the set of positive integers defined as follows:
\[
d_{i1}=
\begin{cases}
  4i-2 &\text{if $n$ is odd},\\
  4i &\text{if  $n$ is even},\\
\end{cases} \quad\text{and}\quad
d_{i2}= 4i+1.
\]
For the case $n-2k \equiv 2,3 \pmod{4}$, we replace the last pair $(d_{k-1,1}, d_{k-1,2})$ with 
\[
d_{k-1,1}=(4k+1)n-2k-1
\quad \text{and}\quad
d_{k-1,2}= (4k+1)n-2k+3.\\
\]
{Also, we set $\ol{D}= [2, \ell n-1]\setminus D$.}
\end{defn}

\begin{ex} 
All through this work, we shall use our construction to build explicitly a solution for $\ell=9$ and $n=5$; in this case we have $D=\{2,5\}$ 
{and $\ol{D}=\{3,4\}\cup[6,44]$}.
\end{ex}

Let us now look at how to prove Theorem \ref{mainth}; we need to
produce a set $\cal B$ of base cycles as required in Theorem \ref{BR}, so we  
will build $(\ell -1)/2$ {transversal} long cycles and $n$ {transversal} short cycles whose list of differences will cover   $\Z_{\ell(2\ell n+1)}\setminus\{0\}$ exactly once.

 First, in Section \ref{lc} we shall construct a set $\cal L$ consisting of  $(\ell -5)/2$, that is  all but two, of the {transversal} long cycles of length 
$2\ell n+1$. The list of differences provided by $\cal L$ will be 
\begin{equation}\label{diff1}
\Delta {\LL} =  \;  (\Z_{2\ell n+1}^* \times (\Z_\ell \setminus \{0,\pm1,\pm2\})) \ \cup \
                     \{(d,f(d)) \;|\; d \in D \ \cup \ -D\}
\end{equation}
with  $f(d)=\pm 1$ for any $d \in D \ \cup \ -D$. 
We point out that the differences not covered by 
$\mathcal{L}$ are chosen in such a way as to facilitate the construction of the {transversal} short cycles.

Then in Section \ref{sc} we shall construct 
a set $\cal S$ of $n$ transversal short cycles such that
\begin{equation}\label{diff2}
    \Delta \mathcal{S} = (\{0\}\times \Z_{\ell}^*) \ \cup \ \{(x, \varphi(x)) \;|\; x \in Z_{2\ell n+1}^-\setminus (D \ \cup \ -D)\}
\end{equation}
 where $\varphi: \Z_{2\ell n+1}^-\setminus (D\ \cup \ -D) \rightarrow \{\pm 1, \pm 2\}$ is a map with some additional properties. More precisely,
\begin{enumerate} 
\item we construct (Lemma \ref{skolem}) a set
$\mathcal{A}=\{A_1,\dots,$ $A_{n-{2k}}\}$ of 
 $\ell$-cycles and a set
$\mathcal{B}=\{B_1,\ldots,B_{2k}\}$ 
of $(\ell-1)$-cycles with vertices in $\Z_{2\ell n +1}$ such that 
$\Delta \mathcal{A} \ \cup\ \Delta \mathcal{B} = \Z_{2\ell n+1}^-\setminus (D\ \cup \ -D)$. 
The construction of the set $\cal A$ requires the crucial use of Skolem sequences.
Also, the cycles in $\mathcal{B}\setminus\{B_1\}$ will have a particular structure called alternating 
(see Definition \ref{alternating}).
\item we  ``lift'' (Proposition \ref{shortcycles}) the cycles  in $\cal A \ \cup \ \cal B$ to obtain a set
$\cal S$ of {transversal} 
$\ell$-cycles with vertices in
$\Z_{2\ell n +1} \times \Z_{\ell}$; to lift these cycles to $K_{\ell(2\ell n+1)}$ means first
to add a vertex to the cycles in $\cal B$ (so that they become 
$\ell$-cycles) and then to add a second coordinate to the vertices of each cycle, in such a way that $\Delta \cal S$ 
satisfies \eqref{diff2}.
\end{enumerate} 

Let us point out that in Proposition \ref{shortcycles} we construct a set of {transversal}
short cycles whose list of differences covers, in particular,
$\{0\}\times \Z^*_{4k+1}$. More precisely, the differences of the form 
$\pm(0,1)$, $\pm(0,2), \ldots, \pm(0,2k)$ are covered by $2k$ distinct {transversal} short cycles. 
Thus our construction requires at least $2k$ 
{transversal} short cycles, and this is the reason why in Theorem \ref{mainth} we {require}
$n\geq 2k=(\ell-1)/2$.
 
To complete the set of base cycles, we need to provide two more {transversal} long cycles $C$ and $C'$.
The construction of these two missing $(2\ell n+1)$-cycles is described in Section \ref{main}; clearly 
$C$ and $C'$ are built in so as to cover the set of the remaining differences, that is, the union of the 
following disjoint sets:
\[\Z_{2\ell n+1}^*\times\{0\}; \quad
\{\pm1,\pm\ell n\}\times \{\pm1,\pm2\};\quad
\bigcup_{i\in\Z_{2\ell n+1}^-} \{i\}\times (\{\pm1, \pm2\}\setminus\{F(i)\}).
\]
where $F:\Z_{2\ell n+1}^- \rightarrow \{\pm 1, \pm2\}$ is the map obtained by glueing together the maps $f$  and  
$\varphi$, that is, $F(x)=f(x)$ or $F(x)=\varphi(x)$ according to whether $x\in D \ \cup \ -D$ or not.
For example, to cover the differences of the form $(i,-F(i))$, the cycle $C$ will contain the path
$P=(c_1,c_2, \ldots, c_{\ell n-1})$ where $c_{i} = (x_i,y_i)$ with
$(x_1, x_2,\ldots, x_{\ell n-1}) = (1,-1,2,-2,3,-3, \ldots)$ and
$(y_1, y_2,\ldots, y_{\ell n-1}) = (1, 1+F(2), 1+F(2)-F(3), 1+F(2)-F(3)+F(4), \ldots)$.
The reader can check that $\Delta P = \pm\{(i,-F(i))\;|\;2\leq i\leq \ell n-1)\}$.\\
In order to obtain the remaining differences, we need $y_{\ell n -1} \equiv 1\; \pmod{\ell}$, 
and since $y_{\ell n -1} = 1+ \sum F$ with $\sum F = \sum_{j=2}^{\ell n -1} (-1)^jF(j)$, we must have 
$\sum F \equiv 0\;\pmod{\ell}$. Recall that $F$ is obtained from the maps $f$ and $\varphi$; we will show that 
$\sum F \equiv 0\;\pmod{\ell}$ by  choosing suitably the map $\varphi$, relying on the fact, 
proved in Proposition \ref{shortcycles}.1, that the alternating sums for $\varphi$ can take any value in $\Z_\ell$.

Putting it all together, the set ${\cal B}= {\cal L} \ \cup \ {\cal S} \ \cup\ \{C, C'\}$ 
is a set of base cycles  as required  in Theorem \ref{BR}.

\section{Transversal long cycles}\label{lc}
In this section we construct (Proposition \ref{longcycles}) a set $\mathcal{L}$ of $(\ell-5)/2$ transversal long cycles whose list of differences $\Delta \mathcal{L}$ has no repeated elements. All 
{the remaining differences}
not lying in  $\Delta \mathcal{L}$ will be covered by the
two {transversal} long cycles given in Section \ref{main} together with the {transversal} short cycles constructed in Section \ref{sc}.

The {transversal} long cycles that we are going to build will be obtained as union of paths with specific vertex-set
and list of differences. More precisely,
given four integers $a\leq b < c \leq d$  with $|(b-a)-(d-c)|\leq 1$, we denote by 
$\PP([a, b], [c,d])$ the family of all paths $P$ with vertex-set 
$[a,b]\ \cup \ [c,d]$ which satisfy the following properties:
\begin{enumerate}
  \item[(1)] for each edge $\{u,u'\}$ of $P$, with $u<u'$, we have that 
  $u\in[a, b]$ and $u'\in[c,d]$;
  \item[(2)] $\Delta P = \pm [c-b, d-a]$.
\end{enumerate}
In other words, $\PP([a, b], [c, d])$ contains all paths whose {two partite sets} are 
$[a, b]$ and $[c, d]$, and 
having the interval $\pm [c-b, d-a]$ as list of differences.
We point out {the connection to $\alpha$-labelings:} a member of $\PP([0, b], [b+1, d])$ can be seen 
as an {\it $\alpha$-labeling} of a path on $d+1$ vertices
(see \cite{Ga13}).  

The following lemma generalizes a result by Abrham \cite{Ab93} concerning the
existence of an $\alpha$-labeling of a path with a given end-vertex.

\begin{lem}\label{paths}
Let  $a\leq b < c \leq d$ with $|(b-a)-(d-c)| \leq 1$; also, set $\gamma_1=b-a$ and $\gamma_2=d-c$.
  Then, there exists a path 
  of $\PP([a,b], [c,d])$ with end vertices $w$ and $w'$ for each of the following values of the pair $(w,w')$:
\begin{enumerate}
\item $(w,w')=(a+i, b-i)$ 
if $\gamma_1 = \gamma_2+1$ and $i\in[0,\gamma_1]\setminus\{\gamma_1/2\}$;
\item $(w,w')=(a+i, c+i)$  if $\gamma_1= \gamma_2$ and $i\in[0,\gamma_1]$;
\item $(w,w')=(c+i, d-i)$ if $\gamma_1 = \gamma_2-1$ and $i\in[0,\gamma_2]\setminus\{\gamma_2/2\}$.
\end{enumerate}            
\end{lem}
\begin{proof} 
Set $I_1=[0, \gamma_1]$ and $I_2=[\gamma_1+1, \gamma_1+\gamma_2 + 1]$.
The existence of a path $P\in\PP(I_1, I_2)$ 
whose end-vertices $u_i$ and $v_i$ satisfy the assertion is proven in \cite{Ab93} 
by using the terminology of $\alpha$-labelings
(see also \cite{Tr13}).
 
Now, let $f:V(P)\rightarrow \Z$ be the map defined as follows: 
\[
  f(x)=
  \begin{cases}
    x+a & \text{if $x \in I_1$},\\
    x+c-(b-a+1) & \text{if $x\in I_2$}
  \end{cases}
\]
and let $Q=f(P)$ be the path obtained from $P$ by replacing each vertex,
say $x$, with $f(x)$. It is clear that
the {two partite sets} of
$Q$ are $f(I_1)=[a,b]$ and $f(I_2)=[c,d]$.
By recalling the properties of a path in $\PP(I_1, I_2)$, we have that
for any $d\in [1, \gamma_1+\gamma_2 + 1]$, there exists an edge $\{u,u'\}$ of $P$, 
with $u\leq \gamma_1<u'$ such that $u'-u=d$. 
By construction, the edge 
$\{f(u), f(u')\}=\{u+a, u'+c -(b-a+1) \}$ lies in $Q$ and it gives rise to the differences 
$\pm(d+c-b-1)$; hence,
 $\Delta Q = \pm [c-b, d-a]$. In other words, $Q\in\PP([a,b],[c,d])$ and it is not difficult to check that its end-vertices, $f(w_i)$
and $f(w'_i)$, satisfy the assertion.
\end{proof}

We use the above result to construct pairs of $(2t+1)$-cycles with vertices in $\Z\times\Z$ satisfying the conditions of the following lemma.

\begin{lem}\label{longcy}
Let $d_1, d_2$, and $t\geq2$ be integers having the same  
parity, with $d_1, d_2 \in [1,t-1]\setminus
\{\lceil t/2\rceil \}$. For any $x,y\in \Z$,  there exist two $(2t+1)$-cycles $C_1$ and $C_2$ with vertices in $\Z \times \Z$
such that
\begin{enumerate}
\item[$(i)$] the projection of $V(C_i)$ on the first components is $[0,2t]$ for $i=1,2$,
\item[$(ii)$] $\Delta C_1 \ \cup \ \Delta C_2 =  \;  \pm ([1, 2t] \times \{x,y\}) \ \cup \  
                     \ \{\pm (d_1,x-y), \pm (d_2,y-x)\}.$                   
\end{enumerate}
 \end{lem}
\begin{proof} 
Let $d_1, d_2,t,x,y$ be integers satisfying the assumptions. 
Also, set $J=[1,2t]\setminus I$ where $I=[1,t]$ or $[1,t+1]$ according to whether $t$ is odd or even.
We are going to show that for any integer $d \in [1,t-1]\setminus
\{\lceil t/2\rceil \}$ with $d \equiv t \pmod{2}$, there exists a cycle $C(d)$ with vertices in $\Z \times \Z$ satisfying the following two properties:
\begin{enumerate}
\item[$(1)$] the projection of $V(C(d))$ on the first components is $[0,2t]$,
\item[$(2)$] $\Delta C(d)  =  \;  \pm (I\times \{x\}) \ \cup \ \pm (J\times \{y\}) \ \cup \ 
                     \ \{\pm (d,y-x)\}.$                   
\end{enumerate}
By setting $C_1=C(d_1)$ and letting $C_2$ be the cycle obtained from $C(d_2)$ by exchanging the roles of 
$x$ and $y$, we obtain a pair $\{C_1, C_2\}$ of cycles which clearly satisfy the assertion.
 
We first deal with the case $t\geq3$ odd. Thus, set $t=2m-1$ and let $d=2k+1$ 
with $0\leq k \leq m-2$.\\
By Lemma \ref{paths} with $[a,b]=[m,2m-2]$, 
  $[b,c]=[2m-1,3m-2]$ and $i=k$, we obtain a path 
  $U=(u_1, \ldots, u_{2m-1})$ of $\PP([a,b], [c,d])$ whose end-vertices are 
  $u_1=3m-2-k$ and $u_{2m-1}=2m-1+k$.
  Again by Lemma \ref{paths} with $[a,b]=[0,m-1]$, $[c,d]=[3m-1,4m-2]$ and $i=k$, there exists
  a path $W=(w_0, w_1, \ldots, w_{2m-1})$ of $\PP([a,b], [c,d])$ whose end-vertices are
  $w_0=k$ and $w_{2m-1}= 3m-1+k$.  
  Now, let $U'$ and $W'$ be the following paths:
  \begin{align*}
    & U' = ((u_1,x),(u_2,0),(u_3,x), \ldots, (u_{2m-3},x),(u_{2m-2},0),(u_{2m-1},x)),\\
    & W' = ((w_0,0),(w_1,y), (w_2,0), \ldots, (w_{2m-3},y), (w_{2m-2},0), (w_{2m-1},y)).
  \end{align*}
  and let $C(d)=U' \ \cup \ ((w_{2m-1},y), (u_1,x)) \ \cup \ W_y \ \cup \ ((u_{2m-1},x), (w_0,0))$ be the $(4m-1)$-cycle  obtained by joining $U'$ and $W'$. Since the projection of $V(C(d))$ on the first coordinates is $V(U) \ \cup V(W)=[0,2t]$, condition (1) holds.\\
Now note that $u_i - u_j>0$ ($w_i -w_j>0$) for any $i$ odd and $j$ even.
Therefore,  by recalling the list of differences of a path in $\PP([a, b], [c,d])$, we have that
  \begin{align*}
      \Delta U' = \; \pm ([1, 2m-2] \times \{x\}),
 \quad\text{and} \quad
        \Delta W' = & \; \pm ([2m,  4m-2] \times \{y\}).
  \end{align*}
  Since $\Delta C(d) = \Delta U' \ \cup \ \Delta W' \ \cup \ \{\pm (w_{2m-1}-u_1, y-x), \pm (u_{2m-1}-w_0,x)\}$ where $w_{2m-1}-u_1 = d$ and $u_{2m-1}-w_0 = 2m-1$,  then (2) is satisfied.

  We proceed in a similar way when $t\geq 2$ is even. So, let $t=2m$ and let
  $d = 2k$ with $1\leq k \leq m-1$  and  $k \neq m/2$.\\
  We apply Lemma \ref{paths} with $[a,b]=[m,2m-1]$, 
  $[c,d]=[2m,3m]$ and $i=k$ to obtain a path 
  $U=(u_1, \ldots, u_{2m+1})$ of $\PP([a, b], [c,d])$ whose end-vertices are 
  $u_1=3m-k$ and $u_{2m+1}=2m+k$.  
  Again, we apply Lemma \ref{paths} with $[a,b]=[0,m-1]$, $[c,d]=[3m+1,4m]$ and $i=k-1$ and obtain
  a path $W=(w_0, w_1, \ldots, w_{2m-1})$ of $\PP([a, b],[c,d])$ whose end-vertices are
  $w_0=k-1$ and $w_{2m-1}= 3m+k$. 
  Now, let $U'$ and $W'$ be the following paths:
  \begin{align*}
    & U' = ((u_1,x), (u_2,0),(u_3,x), \ldots, (u_{2m-3},x), (u_{2m-2},0), (u_{2m+1},x)),\\  
    & W' = ((w_0,0),(w_1,y),(w_2,0), \ldots, (w_{2m-3},y),(w_{2m-2},0),(w_{2m-1},y)).
  \end{align*}
  and let $C(d)=U' \ \cup \ ((w_{2m-1},y), (u_1,x)) \ \cup \ W' \ \cup \ ((u_{2m+1},x), (w_0,0))$ be the $(4m+1)$-cycle  obtained by joining $U'$ and $W'$. As before, one can check that $C(d)$ satisfies (1) and (2). This completes the proof.
\end{proof}
\begin{ex} \label{exlongcy} Let $(d_1, d_2,t) = (1,43,45)$, and let $(x,y)=(3,4)$.
Below are two $91$-cycles $C_1$ and $C_2$ satisfying Lemma \ref{longcy}:

\begin{scriptsize}
\vspace{-.3cm}
\begin{align*}
C_{1} & = ((0,0),(90,3),(1,0),(89,3), \dots, (i,0), (90-i,3), \ldots, (22,0),(68,3),\\
 &(67,4), (23,0),(66,4),\dots,(67-i,4),(23+i,0), \ldots, (46,4),(44,0),(45,4)), \\
C_{2} & =  ( (21,0),(69,4),(22,0),(68,4), \dots, (21-2i,0),(69+2i,4),(22-2i,0),(68+2i,4), \ldots, \\
 & (3,0), (87,4), (4,0), (86,4), (1,0), (88,4), (2,0), (90,4), (0,0), (89,4),\\
 & (46,3),(43,0),(45,3),(44,0),\dots, (46+2i,3),(43-2i,0),(45+2i,3),(44-2i,0), \ldots, \\
 & (62,3), (27,0), (61,3), (28,0),
 (64,3), (24,0), (65,3), (26,0), (63,3), (25,0), (67,3), (23,0), (66,3))
\end{align*}
\end{scriptsize}\noindent
It is straightforward to check that $\Delta C_1 = \pm ([1,45]\times\{4\}) \ \cup \
\pm ([46,90]\times \{3\}) \ \cup \ \{\pm (1,-1)\}$ and $\Delta C_2 = \pm ([1,45]\times\{3\}) \ \cup \ 
\pm ([46,90]\times \{4\}) \ \cup \ \{\pm (43,1)\}$.
\end{ex}

We are now able to construct the required set ${\cal L}$ of $(\ell-5)/2$ transversal long cycles; {we shall use the set $D$ of definition \ref{setD}.}

\begin{prop}\label{longcycles} For $\ell \equiv 1 \pmod{4}$ and $\ell\geq 5$ 
there exists a set $\LL$ of $(\ell-5)/2$ transversal $(2\ell n+1)$-cycles of 
$K_{\ell(2\ell n+1)}$ such that
\[\Delta {\LL} =  \;  (\Z_{2\ell n+1}^* \times (\Z_\ell \setminus \{0,\pm1,\pm2\})) \ \cup \
                     \{(d,f(d)) \;|\; d \in D \ \cup \ -D\}
\]
where $f(d)=\pm 1$ for any $d \in D \ \cup \ -D$.
\end{prop}
\begin{proof}
We shall construct a set of transversal long cycles $\LL$ by applying repeatedly Lemma \ref{longcy}.
We will use the following straightforward property of an integer $d \in D$: 
$d \equiv 1,2 \pmod{4}$ for $n$ odd, and $d \equiv 0,1 \pmod{4}$ for $n$ even.

Set $t=\ell n$, $k=(\ell-1)/4$ and let $D=\{d_{i1}, d_{i2} \;|\; i=1,\ldots,k-1\}$. For any $d\in D$ we define the integer $d'$ as follows: 
\[
\text{$d' = d$ when $d\equiv 0,2 \pmod{4}$, and  $d'=2t+1-d$ when $d\equiv 1 \pmod{4}$,}
\]
and set $D' = \{d' \;|\; d\in D\}$.
Of course, $D \ \cup \ -D \equiv D' \ \cup \ -D' \pmod{2t+1}$; also, all integers in $D'$ are congruent to $0$ or $2 \mod{4}$  according to whether $t$ is even or odd; it then follows that all integers in $D'/2$ have the same parity as
$t$. Recalling how $D$ is defined (Definition \ref{setD}), 
it is easy to check that $D'/2 \subseteq [2, t-1]\setminus\{\lceil t/2 \rceil\}$.

For $i=1,\dots,k-1$ set $x=2i+1,y=2i+2$ and take $d_1={d'_{i1}}/2$ and $d_2={d'_{i2}}/{2}$. In view of the considerations above, all the assumptions of Lemma \ref{longcy} are satisfied. Therefore, there exist two cycles $C_{i1}, C_{i2}$ for $i=1,\dots,k-1$ with vertices in $\Z\times \Z $ such that
\[\text{the projection of $C_{i1}$ (resp. $C_{i2}$) on the first coordinates is $[0,2\ell n]$}, \text{and}
\]
\[\Delta C_{i1} \ \cup \ \Delta C_{i2} =  \;  \pm \left([1,2\ell n] \times [2i+1,2i+2]\right)  \
                     \cup \ \{\pm (d'_{i1}/2,-1), 
                              \pm (d'_{i2}/2,1)\}.\]
Now, consider the vertices of $C_{i1}$ and $C_{i2}$ (and also their differences), as elements of $\Z_{2\ell n+1} \times \Z_{\ell}$. We denote by $\psi$ 
the permutation of $\Z_{2\ell n+1}\times \Z_{\ell}$ defined as 
$\psi(a,b)=(2a,b)$, and set $C'_{ij}=\psi(C_{ij})$ for $i\in[1,k-1]$ and $j=1,2$.
It is clear that $C'_{ij}$ is a transversal $(2\ell n+1)$-cycle of $K_{\ell(2\ell n+1)}$; moreover,
\[\Delta C'_{i1} \ \cup \ \Delta C'_{i2} =  \;   \pm \left(\Z^*_{2\ell n+1} \times \{2i+1,2i+2\}\right)   \
                     \cup \ \{\pm (d'_{i1},-1), 
                              \pm (d'_{i2},1)\}.\]
Finally, set ${\LL}=\{C'_{i1}, C'_{i2}\;|\; i\in[1,k-1]\}$; of course, 
$\Delta {\LL} = \bigcup_i\Delta \{C'_{i1}, C'_{i2}\}$.
By recalling how the $d'_{ij}$s are defined, we get
\begin{align*}
  \Delta {\LL} &=  
 (\Z_{2\ell n+1}^* \times \Z^-_\ell) \ \cup \
                    \{\pm(d_{i1},f(d_{i1})) \;|\; i\in[1,k-1], j=0,1\}.
\end{align*}
where $f(d_{ij})=1$ or $-1$ for any $i\in[1,k-1]$ and $j=1,2$. This shows that $\LL$ is the required set of transversal long cycles.
\end{proof}
\begin{ex} \label{exlongcycle}
Let $\ell=9$ and $n=5$; then $k=2$ and the set $D$ is $\{2,5\}$. In this case, the set 
${\cal L}$ of Proposition \ref{longcycles} consists of two {transversal} long cycles, ${\cal L}=\{C_1', C_2'\}$, that we will construct by following 
the proof. Note that $D'=\{2,86\}$, and hence $D'/2=\{1,43\}$. 

Now, set $(d_1, d_2, t) = (1,43,45)$, $(x,y)=(3,4)$ and apply Lemma \ref{longcy} to obtain two {transversal} $91$-cycle  $C_1$ and $C_2$: for example,  we will take the cycles of Example \ref{exlongcy}. 
To obtain $C_1'$ and $C_2'$, it is enough to multiply by $2$ the first component 
in the cycles $C_{1}$ and $C_{2}$ and reduce modulo $91$, that is, 
\begin{small}
\vspace{-.4cm}
\begin{align*}
\small
C'_{1} & = ((0,0),(89,3),(2,0),(87,3), \ldots, (45,3), (43,4) , \ldots, (1,4),(88,0),(90,4)), \\
C'_{2} & =  ((42,0),(47,4),(44,0),(45,4), \dots, (87,4),(1,3), \ldots, (43,3), (46,0), (41,3)).
\end{align*}
\end{small}%
Taking into account the list of differences of $C_1$ and $C_2$
(Example \ref{exlongcy}), we can see that
$\Delta {\cal L} = \Z_{91}^* \times \{\pm3,\pm4\} \ \cup \ \{\pm(2,-1)), \pm (5,-1)\}.$
\end{ex}

\section{Transversal short cycles}\label{sc}
In this section we show that there exists a set $\mathcal{S}$
of $n$ transversal short cycles for all $n \geq (\ell-1)/2$, whose list of differences is disjoint from $\Delta \mathcal{L}$, where $\mathcal{L}$ is the set of 
{transversal} long cycles constructed in Section \ref{lc}.  

We first provide a set $\mathcal{A}$ of $n-2k$ cycles of length $\ell=4k+1$ and a set $\mathcal{B}$ of $2k$ cycles of length $\ell-1$ (Lemma \ref{skolem}), with vertices in $\Z_{2\ell n+1}$. After that, in Proposition \ref{shortcycles} we first inflate the cycles in $\mathcal{B}$ to get $\ell$--cycles; then, we lift the cycles in $\mathcal{A}$ and $\mathcal{B}$ to $\Z_{\ell(2\ell n+1)}$ by adding a second coordinate to the vertices of the cycles 
(this will be done 
by following Lemmas \ref{pathi} and \ref{pathx}). All these steps will give us
the required set $\mathcal{S}$. 

\subsection{Building the short cycles in $\Z_{2\ell n +1}$}

First, we need the following definition which describes the structure of the cycles in the set $\mathcal{B}$ 
we shall build in Lemma \ref{skolem}.

\begin{defn}\label{alternating} Let $\ell=4k+1$ and let
  $B=(b_0=0, b_1, \ldots, b_{4k-1})$ be a $4k$\-cycle 
  with vertices in {$[-\ell n, \ell n]$}.
For any $i\in[1,4k-1]$ set $\delta_i = (-1)^i (b_i-b_{i-1})$ 
  and $\delta_{4k} = b_0-b_{4k-1}$. Then $B$ is said to be {\it alternating} if the following conditions 
  are satisfied:
  \begin{enumerate}
    \item $\delta_i \in [1,\ell n]$ for any $i\in[1,4k]$;
    \item $\delta_i \equiv i+1 \pmod{2}$ for $i \in [1,2k]$;
    \item $\delta_i \equiv i \pmod{2}$ for $i \in [2k+1,4k]$.     
  \end{enumerate}
\end{defn} 
\begin{ex}\label{exalternating} 
Let $\ell=9$, $n=5$. Denote by $B=(b_0, b_1, \ldots, b_7)$ the $8$-cycle with vertices in
$[-45,45]$ defined as follows:
$B= (0,-30,1, -31, 2, -33, 3,-34)$. 
Now, note that 
$(\delta_1, \delta_2, \ldots, \delta_8)=(30,31, 32, 33, 35, 36, 37, 34)$; also each $\delta_h$ lies in $[1,45]$. Therefore, 
conditions 1, 2, and 3 of the above definition are satisfied, and hence $B$ is alternating.

Note that $-B = (0,10,-1, 11, -2, 13, -3,14)$ has the same list of differences as $B$, but it is not alternating since
conditions 2 and 3 are not satisfied.
\end{ex}

We start building the short cycles. First, we need a result providing a sufficient condition for a set of $(\ell-1)$ positive integers $U$ to be obtained as the list of differences of an $(\ell-1)$-cycle, possibly alternating.

\begin{lem}\label{4kgons} Let $\ell \equiv 1 \;\pmod{4}$ and 
  let $U\subseteq[1,\ell n]$ be  a set of size $(\ell-1)$.
  If $U$ can be partitioned into pairs of consecutive integers, then there exists an 
  $(\ell-1)$--cycle $C$ with $V(C)\subseteq  \Z_{2\ell n+1}$ and $\Delta C = \pm U$.
  
  Also, if $U=[u,u+\ell-2]$ with $u$ even, then $C$ is alternating.
\end{lem}
\begin{proof} Let $\ell = 4k+1$ and let $(\delta_1, \delta_2, \ldots, \delta_{2k}, x, \delta_{2k+1}, \ldots, \delta_{4k-1})$ be
the increasing sequence of the integers in $U$. By assumption, $U$ can be partitioned into pairs 
of consecutive integers, that is, $\delta_{2i-1}+1 = \delta_{2i}$, $x+1=\delta_{2k+1}$, and 
$\delta_{2j}+1=\delta_{2j+1}$, for $i\in [1,k]$ and $j \in [k+1,2k-1]$. 
Therefore, it is not difficult to check that
$x = -\sum_{i=1}^{4k-1}(-1)^{i} \delta_{i}$.
  
  Now, set $b_0=0$ and $b_i = \sum_{h=1}^{i}(-1)^{h}\delta_{h}$ for $i\in [1,4k-1]$. 
  Since the sequence $(\delta_1, \delta_2, \ldots, \delta_{4k-1})$ is increasing, 
  it is straightforward to check that all $b_i$s are pairwise distinct. 
  Then, we can consider the $4k$-cycle  $C=(b_0, b_{1},\dots,b_{4k-1})$. 
  Note that $\delta_i = (-1)^i(b_i-b_{i-1})$ for $i \in [1, 4k-1]$. Also, $b_{4k-1} = -x$ hence, 
$x = b_0-b_{4k-1}$. It then follows that 
$\Delta C = \pm U$. 

The final part of the assertion follows by taking into account that $(\delta_1, \ldots,$ $\delta_{2k},$ $x, \delta_{2k+1}, \ldots, \delta_{4k-1})$ is
the increasing sequence of the integers in $U$.
\end{proof}

As in \cite{BuDa,BuRi05}, we also need the crucial use of {\it Skolem sequences}.
A Skolem sequence of order $n$ can be viewed as a sequence
$S=(s_1,\dots,s_n)$ of positive integers such that $\bigcup_{i=1}^n\{s_i,s_i+i\}=\{1,2,\dots,2n\}$
or $\{1,2,\dots,2n+1\}\setminus\{2n\}$. One speaks of an {\it ordinary Skolem sequence} in the first 
case and
of a {\it hooked Skolem sequence} in the second.
It is well known (see \cite{Sha07}) that there exists a Skolem sequence of order $n$ for every positive integer $n$; it is ordinary for $n\equiv0$ or $1 \pmod{4}$  and hooked for $n\equiv2$ or $3 \pmod{4}$.

\begin{lem}\label{skolem} For $\ell=4k+1$ and $n\geq 2k$, there exist a set
$\mathcal{A}=\{A_1,\dots,$ $A_{n-{2k}}\}$ of $\ell$-cycles and a set
$\mathcal{B}=\{B_1,\ldots,B_{2k}\}$ 
of  $(\ell-1)$-cycles with vertices in $\Z_{2\ell n +1}$ satisfying the following properties:
\begin{enumerate}
 \item $\Delta \mathcal{A} \ \cup\ \Delta \mathcal{B} = \Z_{2\ell n+1}^-\setminus (D\ \cup \ -D)$;
 \item for $i\geq 2$, the cycle $B_i$ is alternating and
 $\Delta B_i=\pm [u_i,u_i+4k-1]$, with $u_i=2(i-1)n$.
\end{enumerate}
\end{lem}
\begin{proof} In the proof we will start by constructing the cycles of $\cal A$ so that
the set of integers in $\Z_{2\ell n+1}^-\setminus(D \ \cup \ -D)$ not covered by $\Delta \cal A$ can be split up into even-length intervals.
A suitable choice of $4k$-gons as in Lemma \ref{4kgons} will then make up the set $\cal B$ and take care of the remaining differences.

Let us fix a Skolem sequence $S=(s_1,...,s_{n-2k})$ of order $n-2k$.
So $S$ is ordinary for $n-2k \equiv 0$ or $1$ $\pmod{4}$ and hooked otherwise.

We start with the non-hooked case, so let $n\equiv2k$ or $2k+1 \pmod{4}$.

Let us construct the cycles of $\cal A$; for $1\leq i\leq n-2k$, let $A_i=(a_{i0}, a_{i1}, \ldots,$ $a_{i,4k})$,
where:
\[
a_{ij}=
\begin{cases}
  (4k-2-j)n & \text{for  $1 \leq j \leq 4k-3$, $j$ odd},\\
  jn+i-2k    & \text{for  $0 \leq j \leq 2k-2$, $j$ even}, \\
  jn+i-1+2k    & \text{for $2k \leq j \leq 4k-2$, $j$ even}, \\
  -2k & \text{for $j=4k-1$},\\
  s_i+i+(4k-1)n-1 &\text{for $j=4k$}.
\end{cases}
\]
It is tedious but straightforward to check that
$\Delta {\cal A}  = I_0 \cup I_1  \dots \cup I_{2k-1}$
where
\[I_\alpha=
\begin{cases}
 [4k, 2n-1] +2n\cdot \alpha & \text{for $\alpha=0, \ldots, 2k-2$} \\
 [(4k-2)(n+1)+2,(4k+1)n-2k-1]  & \text{for $\alpha=2k-1$}.
\end{cases}
\]
We now apply Lemma \ref{4kgons} to  construct $\cal B$ such that 
$\Delta {\cal B}=\Z_{2\ell n+1}^-\setminus (\pm D \ \cup \ \Delta{\cal A})$.
Let $J_{\beta}$ be the interval between $I_{\beta-1}$ and $I_{\beta}$ for $1\leq\beta\leq 2k-1$;
each such $J_\beta$ has even length $4k$.
Also, set $J_{0}=[2,4k-1]$ and $J_{2k}=[(4k+1)n-2k,(4k+1)n-1]$. Note that the $I_{\alpha}$'s and $J_{\beta}$'s are pairwise disjoint and cover altogether the integers from $2$ to $(4k+1)n-1$, namely:
\[
  [2,(4k+1)n-1]= \bigcup_{\alpha=0}^{2k-1}I_{\alpha}\ \cup \ \bigcup_{\beta=0}^{2k}J_{\beta}.
\]
It is easy to check that $D\subseteq J_0$ and $J_{0} \setminus D$ is a set of $k$ pairs of disjoint consecutive integers. In view of Lemma \ref{4kgons} there exists $2k$ cycles $C_{0}, C_{1}, \ldots, C_{2k-1}$ of length  $4k$ and vertices in $\Z_{2\ell n+1}$ whose lists of differences are the following:
\[\Delta C_{\beta}=
\begin{cases}
 \pm (J_{0}\setminus D) \ \cup \ \pm J_{2k}& \text{for $\beta=0$} \\
 \pm J_{\beta} & \text{for $1\leq \beta\leq2k-1$}.
\end{cases}
\]
Note that  for any $\beta \in [1,2k-1]$ the smallest integer in $J_{\beta}$ is even, hence $C_{\beta}$ is alternating by
Lemma \ref{4kgons}.


We now consider the hooked case, so let $n\equiv2k+2$ or $2k+3 \pmod{4}$.

Let $\cal A$ be the cycle--set constructed earlier.
Since now the Skolem sequence is hooked, then the list of differences of $\cal A$ has the following form:
$
\Delta {\cal A}  = I_0 \ \cup \ I_1 \dots \ \cup \ I_{2k-2} \ \cup \ I^*_{2k-1}
$
where $I_\alpha=[4k, 2n-1] +2n\cdot \alpha$ for $\alpha=0, \ldots, 2k-2$ and
\[
 I^*_{2k-1} = [(4k-2)(n+1)+2,(4k+1)n-2k-2]  \ \cup \ \{(4k+1)n-2k\}.
\]
We apply Lemma \ref{4kgons} to  construct $\cal B$ such that 
$\Delta {\cal B}= \Z_{2\ell n+1}^-\setminus (\pm D \ \cup \ \Delta{\cal A})$.
Let $J_{\beta}$ be the interval between $I_{\beta-1}$ and $I_{\beta}$ for $1\leq\beta\leq 2k-1;$ 
each such $J_\beta$ has even length $4k$.
Also, set $J_{0}=[2,4k-1]$ and $J_{2k}=[(4k+1)n-2k-1,(4k+1)n-1]\setminus \{(4k+1)n-2k\}$.
Note that the $I_{\alpha}$'s and $J_{\beta}$'s are pairwise disjoint and cover altogether the integers from $2$ to $(4k+1)n-1$, namely:
\[
  [2,(4k+1)n-1]= \bigcup_{\alpha=0}^{2k-1}I_{\alpha}\ \cup \ \bigcup_{\beta=0}^{2k}J_{\beta}.
\]

It is easy to check that $D\subseteq J_0 \ \cup \ J_{2k}$ and that $J_{0} \setminus D$ is a set of $k+1$ pairs of disjoint consecutive integers. Also,
$J_{2k}\setminus D=[(4k+1)n-2k+1, (4k+1)n-2k+2]\ \cup \ [(4k+1)n-2k+4, (4k+1)n-1]$; then, $J_{2k}$ is the union of a $2$--set and a $(2k-4)$--set both made of consecutive integers.
In view of Lemma \ref{4kgons} there exist $2k$ cycles $C_{0}, C_{1}, \ldots, C_{2k-1}$ of length  $4k$ and vertices in $\Z_{2\ell n+1}$ whose lists of differences are the following:
\[\Delta C_{\beta}=
\begin{cases}
 \pm (J_{0}\setminus D) \ \cup \ \pm(J_{2k}\setminus D)& \text{for $\beta=0$} \\
 \pm J_{\beta} & \text{for $1\leq \beta\leq2k-1$}.
\end{cases}
\]
As before, we point out that  for any $\beta \in [1,2k-1]$ the smallest integer in $J_{\beta}$ is even, hence $C_{\beta}$ is alternating by Lemma \ref{4kgons}.

We shall obtain the set $\cal B$ required in the proof by setting, for instance, $B_{i+1}=C_i$ for $i\in [0,2k-1]$.
\end{proof}

\begin{ex}\label{exskolem}
Let  $n=5$ and $\ell=9$, hence $k=2$. 
In this case, $\Z^-_{91}=\pm[2,44]$ and the set $D$ of  Definition \ref{setD} has only two elements: $D=\{2,5\}$.
Below, we provide the sets 
of short cycles ${\cal A} = \{A_1\}$ and ${\cal B} = \{B_1,B_2,B_3,B_4\}$ of Lemma \ref{skolem}.
\begin{align*}
A_1 &=(-3,25,7,15,24,5,34,-4,36),  \\
B_1 &= (0,-3, 1,  -5, 2, -40, 3,-41), \;\;
B_2 = (0,-10,1, -11, 2, -13, 3,-14), \\ 
B_3 &= (0,-20,1, -21, 2, -23, 3,-24), \;\; 
B_4 = (0,-30,1, -31, 2, -33, 3,-34).
\end{align*}
First, note that $\Delta A_1 = \pm\{8,9,18,19,28,29,38,39,40\}$ and 
$\Delta B_1=\pm\{3,4,$
$6,7,41,42,43,44\}$. 
Now, for $i=2,3,4$, let $B_i=(b_{0,i}, b_{1,i}, \ldots, b_{4k-1,i})$, and
set $\delta_{h,i} = (-1)^h (b_{h,i} - b_{h,i})$ for $h\in [1,8]$, with $b_{4k}=b_0$.
It is easy to check that
\[
(\delta_{1,i}, \delta_{2,i}, \ldots, \delta_{4k,i}) =
\begin{cases}
  (10,11, 12, 13, 15, 16, 17, 14) & \text{for $i=2$},\\
  (20,21, 22, 23, 25, 26, 27, 24) & \text{for $i=3$},\\
  (30,31, 32, 33, 35, 36, 37, 34) & \text{for $i=4$}.\\
\end{cases}
\]
It follows that the cycles $B_2, B_3, B_4$ are alternating. Also, 
$\Delta B_2= \pm[10,17]$, $\Delta B_3= \pm[20,27]$, and $\Delta B_4= \pm[30,37]$.
Therefore, $\Delta {\cal A} \ \cup \ \Delta {\cal B} = \Z_{81}\setminus (\pm D)$. 
\end{ex}

\subsection{Lifting the short cycles to $\Z_{\ell(2\ell n+1)}$}
As mentioned at the beginning of this section, to obtain the set $\mathcal{S}$ of transversal short cycles we need to lift the cycles in ${\cal A}$ and $\cal B$
by adding a second component to each vertex.  A special attention is to be devoted to the cycles of $\cal B$. Indeed, 
given an $(\ell-1)$-cycle $B\in \cal B$, say $B=(b_0, b_1, \ldots, b_{4k-1})$, 
we will consider as a lift of $B$ any cycle $B'$ of the following form:  
\[B'=((b_0,p_0), (b_1,p_1), \ldots, (b_{4k-1},p_{4k-1}), (b_{4k}, p_{4k}))\] 
whose vertices are elements of 
$\Z_{2\ell n +1} \times \Z_{\ell}$ and either $b_{4k}=b_{0}$ or $b_{4k}=b_{4k-1}$. 
If $\Delta B=\pm[u, u']$ with $0<u\leq u'$, then $\Delta B'$ has the following form:
\[
  \Delta B' = \{\pm(i, \varphi(i))\;|\; i\in [u,u']\}\ \cup \ 
  \{\pm(0, x)\}
\]
for a suitable map $\varphi: [u,u'] \rightarrow \Z_{\ell}$, and $x=p_{4k} - p_0$ or  $x=p_{4k} - p_{4k-1}$ according to whether $b_{4k}=b_{0}$ or $b_{4k-1}$.
As will become clear later on, we will need to determine the 
\emph{partial sum} $\sum_{i=u}^{u'} (-1)^i\varphi(i)$ related to $B'$.

The following lemma shows that the partial sum related to the lift $B'$ of an alternating cycle $B$ 
depends only on some of the labelings of the second components. More precisely, 

\begin{lem}\label{specialcycles}
Let $\ell=4k+1$ and let
$B=(0, b_{1},\dots,b_{4k-1})$ be a $4k$-cycle with vertices in $\Z_{2\ell n+1}$.
Denote by $B'$ the $\ell$-cycle with vertices in 
$\Z_{2\ell n+1}\times \Z_{\ell}$ obtained from $B$ as follows: 
$B'=((0,0), (b_1,p_1), \ldots, (b_{4k-1},$ $p_{4k-1}),$ $(b_{4k},p_{4k}))$ where $b_{4k}=0$ or $b_{4k-1}$. 

If $B$ is alternating and $\Delta B \ \cap\ [1,\ell n] = [u,u']$, then the map 
$\varphi:[u, u'] \rightarrow \Z_{\ell}$ such that $(i,\varphi(i))\in \Delta B'$
satisfies the relation:
\[
\sum_{i=u}^{u'} (-1)^i\varphi(i) = 
\begin{cases}
  p_{4k}-2p_{2k} & \text{if $b_{4k}=0$}, \\
  p_{4k-1}-p_{4k}-2p_{2k} & \text{if $b_{4k}=b_{4k-1}$}.
\end{cases}
\]
\end{lem}
\begin{proof} For brevity, set $\sum = \sum_{i=u}^{u'} (-1)^i\varphi(i)$. Also, as in Definition \ref{alternating},  for any $h\in[1,4k-1]$ set 
$\delta_h = (-1)^h (b_h-b_{h-1})$, where $b_0=0$, and set $\delta_{4k} = -b_{4k-1}$. Since $B$ is alternating,
  $\delta_h \in [1,\ell n]$ for any $h\in[1,4k]$; also, by assumption, 
  $\Delta B \ \cap\ [1, \ell n] = [u,u']$.
  Therefore,  $[u,u']=\{\delta_1, \delta_2, \ldots, \delta_{4k}\}$ hence, 
  $\sum = \sum_{h=1}^{4k} (-1)^{\delta_h}\varphi(\delta_h)$. 
  
  For any $h \in [1,4k-1]$ we have that $\varphi(\delta_h) = (-1)^h (p_h-p_{h-1})$, where $p_0=0$;
  also, it is not difficult to check that
  \[
    \varphi(\delta_{4k})=
    \begin{cases}
      p_{4k}-p_{4k-1} & \text{if $b_{4k}=0$},\\
      -p_{4k} & \text{if $b_{4k}=b_{4k-1}$}.      
    \end{cases}
  \]
  Now set
  $\sum' = \sum_{h=1}^{2k} (-1)^{\delta_h}\varphi(\delta_h)$  and   
  $\sum'' = \sum_{h=2k+1}^{4k-1} (-1)^{\delta_h}\varphi(\delta_h)$.
  Since by Definition \ref{alternating} we have $\delta_h \equiv h+1 \pmod{2}$ for $h \in [1,2k]$, then
  \[
    {\sum}'   = \sum_{h=1}^{2k} (-1)^{h+1}\varphi(\delta_h) = 
    - \sum_{h=1}^{2k} (p_h-p_{h-1}) = p_0 - p_{2k} = - p_{2k}.
  \]
  On the other hand, $\delta_h \equiv h \pmod{2}$ for $h \in [2k+1,4k]$, therefore
  \[
    {\sum}'' = \sum_{h=2k+1}^{4k-1} (-1)^{h}\varphi(\delta_h) = 
    \sum_{h=2k+1}^{4k-1} (p_h-p_{h-1}) = p_{4k-1} - p_{2k}.
  \]
  and $(-1)^{\delta_{4k}}\varphi(\delta_{4k}) = \varphi(\delta_{4k})$. Since
  $\sum = \sum' + \sum'' +$ $\varphi(\delta_{4k})$, the assertion easily follows.
\end{proof}
\begin{ex}\label{exspecialcycles} 
Let  once more $n=5$ and $\ell=9$ (i.e. $k=2$). Consider the $8$-cycle
$B_4 = (0,-30,1, -31, 2, -33, 3,-34)$ of Example \ref{exskolem}. It is 
alternating and $\Delta B \ \cap\ [1,45] = [30,37]$. Therefore, the assumption of 
Lemma \ref{specialcycles} are satisfied.
Now, denote by $B'$ the cycle with vertices in $\Z_{91} \times \Z_9$ obtained from $B$ by repeating the last vertex and then adding the second components $(p_0, p_1, \ldots, p_8) =  (0,2,3,4,5,7,8,6,1)$, that is,
\begin{align*}
B' = (&(0,0),(-30,2),(1,3),(-31,4),(2,5),(-33,7),
      (3,8),(-34,6),(-34,1)).
\end{align*}
It is easy to check that $\Delta B' = \pm(\{31,33,36\}\times\{1\})
\ \cup \ \pm(\{32,34\}\times\{-1\})
\ \cup \ \pm(\{30,35\}\times\{-2\})
\ \cup \ \{\pm (37,2), \pm (0,4)\}$.
Therefore, we can define the map  
$\varphi:[30,37] \rightarrow \Z_{\ell}$  such that $(i,\varphi(i))\in \Delta B'$. It is easy to check that 
\[
\sum_{i=30}^{37} (-1)^i\varphi(i) = 4 = 
  p_{7}-p_{8}-2p_{4}.
\]
\end{ex}

The following two lemmas tell us how to label the second components of the vertices in the cycles of Proposition \ref{skolem}, in order to get the set $\mathcal{S}$ of short cycles. 

\begin{lem}\label{pathi}
Let $\ell=4k+1$; for any $i \in [1, 2k-1]$ there is an $\ell$-cycle $Q_i=(0,q_1,\dots q_{4k})$ with vertex set $\Z_\ell$ 
such that
$q_j=j$ \text{for} $j=1,2,\dots,4k-i, q_{4k}=-i,$ and  $\Delta Q_i\setminus\{\pm i\}$ contains only $\pm1$s and $\pm2$s.
\end{lem}
\begin{proof}
The assertion is easily verified: consider for instance the cycle
\begin{align*}
Q_i=(&0, 1, 2, \dots, 4k-i, \\
&4k-i+2,4k-i+4,4k-i+6,\dots,4k,\\
&4k-1,4k-3,4k-5,\dots,4k-i+1)
\end{align*}
for $i$ even, and for $i$ odd the cycle
\begin{align*}
  Q_i=(&0,1,2,\dots,4k-i,\\
  &4k-i+2,4k-i+4,4k-i+6,\dots,4k-1,\\
  &4k,4k-2,4k-4,\dots, 4k-i+1).
\end{align*}
\end{proof}
We point out to the reader that the cycles $Q_i$s of the above lemma are unique.
\begin{lem}\label{pathx}
Let $\ell=4k+1$; for any $m\in \Z_\ell$ there is an $\ell$-cycle $P_m=(0,p_1,\dots,p_{4k})$ with vertex set 
$\Z_{\ell}$ such that
\begin{enumerate}
\item for $m \neq \pm 2k$, we have $p_{4k}=2k$ and $p_{4k}-2p_{2k}=m$, 
\item for $m=\pm 2k$, we have $p_{4k}-p_{4k-1}=m$ and $p_{2k} = -m$,
\item for any $m \in \Z_{\ell}$, $\Delta P_m\setminus\{\pm 2k\}$  contains only  $\pm1$s and $\pm2$s.
\end{enumerate}
\end{lem}
\begin{proof} For a clear description of our construction, we
need to work first on the integers. The required cycles will be then obtained by reducing  $\pmod{\ell}$. 

Let $m \in [0,4k]$ with $m\neq 2k,2k+1$ and let $x$ be the integer in $[0, 4k]$ such that $x \equiv (2k-m)/2 \pmod{\ell}$; of course, $x \neq 0,2k$.
 We first work with the case $0<x<2k$. If $x$ is odd, we may take
\begin{align*}
P_m=(&0,4k-1,4k-3,\dots,2k+x+2,\\
&2k+x+1, 2k+x+3, 2k+x+5,\dots, 4k, \\
&1,2,3,\dots,x,x+1,\dots,2k-1,\\
&2k+1,2k+3,2k+5,\dots, 2k+x,\\ 
&2k+x-1,2k+x-3,2k+x-5, \dots, 2k).
\end{align*}
Similarly, for $x$ even, we may take
\begin{align*}
P_m=(&0,4k-1,4k-3,\dots,2k+x+1,\\
&2k+x+2, 2k+x+4, 2k+x+6, \dots, 4k, \\
&1,2,\dots,x,x+1,\dots,2k-1,\\
&2k+1,2k+3,2k+5,\dots,2k+x-1 \\ 
&2k+x,2k+x-2,2k+x-4,\dots, 2k).
\end{align*}
Suppose now that $2k+1\le x \le 4k$, and say $x=2k+x'$. If $x$ is even, we take
\begin{align*}
P_m=(&0,2,4,\dots,x'-2,\\
&x'-1,x'-3,x'-5,\dots, 1,\\
&4k,4k-1,4k-2,\dots,x,x-1,\dots,2k+1,\\
&2k-1,2k-3,2k-5,\dots, x'+1,\\
&x',x'+2,x'+4, \dots, 2k).
\end{align*}
If $x\ne2k+1$ is odd, then
\begin{align*}
P_m=(&0,2,4,\dots,x'-1,\\
&x'-2,x'-4, x'-6, \dots,1,\\
&4k,4k-1,4k-2,\dots, x, x-1, \dots,2k+1\\
&2k-1,2k-3,2k-5,\dots,x', \\
&x'+1, x'+3, x'+5,\dots, 2k),
\end{align*}
and for $x=2k+1$, that is for $m=2k-1$, we take:
\begin{align*}
P_{2k-1}=(&0, 4k, 4k-1, \ldots, 2k+1, \\
&2k-1, 2k-3, 2k-5, \ldots,1,\\
&2,4,6,\ldots, 2k).
\end{align*}
In all the above cases we have $p_{2k}=x$. 
Finally, the cycle $P_{2k}$ is given as follows:
\begin{align*}
P_{2k}=(&0, 2,3,4, \ldots, p_{2k}=2k+1, \\
& 2k+3, 2k+5, 2k+7, \ldots, 4k-1,\\
& 4k,4k-2,4k-4,\ldots, 2k+2, 1),
\end{align*} 
and $P_{-2k} = - P_{2k}$.
It is straightforward to check that after reducing modulo $\ell$, all cycles just defined satisfy the requirements of the lemma.
\end{proof}

\begin{ex} For $\ell=9$ the cycles constructed in the two previous lemmas are as follows.
\begin{scriptsize}
\begin{align*}
Q_1=&(0,1,2,3,4,5,6,7,8)  &Q_2&=(0,1,2,3,4,5,6,8,7)  & Q_3&=(0,1,2,3,4,5,7,8,6)\\
P_0 = & (0,7,8,1,2,3,5,6,4) & P_1&=  (0,1,8,7,6,5,3,2,4) & P_2& =  (0,7,6,8,1,2,3,5,4)\\
P_3= & (0,8,7,6,5,3,1,2,4) &P_4 &=  (0,2,3,4,5,7,8,6,1) & P_5&=  (0,7,6,5,4,2,1,3,8) \\
P_6 = & (0,2,3,1,8,7,6,5,4) & P_7&=  (0,8,1,2,3,5,7,6,4) &P_8 &= (0,2,1,8,7,6,5,3,4).
\end{align*}
\end{scriptsize}
\end{ex}
We are now able to construct the required set $\mathcal{S}$ of $n$ transversal short cycles in such a way that $\Delta \mathcal{S} \ \cup \ \Delta\mathcal{L}$ has no repeated differences.

\begin{prop}\label{shortcycles}
  Let $\ell \equiv 1 \pmod{4}$. For every integer $n\geq (\ell-1)/2$ and for every $s \in \Z_{\ell}$, there exists a set $\mathcal{S}$ of $n$ transversal $\ell$--cycles of $K_{\ell(2\ell n+1)}$ whose list of differences is of the form
  \[
    \Delta \mathcal{S} = (\{0\}\times \Z_{\ell}^*) \ \cup \ \{(x, \varphi(x)) \;|\; x \in Z_{2\ell n+1}^-\setminus (D \ \cup \ -D)\}
  \]
  where $\varphi: \Z_{2\ell n+1}^-\setminus (D\ \cup \ -D) \rightarrow \{\pm 1, \pm 2\}$ is a map such that:
\begin{enumerate}
\item \label{allvalues}
$\sum_{i\in \ol{D}} (-1)^i \varphi(i)=s$, where $\ol{D} = [2,\ell n-1]\setminus D$;
\item \label{uni}
the set $\{i \in \ol{D} \;|\;\varphi(i)=(-1)^{i+1}\}$
has size $\geq (\ell-5)(\ell-1)/4$.
\end{enumerate}
\end{prop}
\begin{proof}
Consider the sets $\cal A$ and $\cal B$, constructed in Lemma \ref{skolem}, containing cycles of $K_{2\ell n+1}$. 
We have to ``lift'' them to cycles of $K_{\ell(2\ell n+1)}$. Since the vertices of the cycles in $\cal A \ \cup \ B$ lie in $\Z_{2\ell n +1}$, while the vertex-set of $K_{\ell(2\ell n+1)}$ has been identified with 
$\Z_{2\ell n +1} \times \Z_{\ell}$, to lift these cycles to $K_{\ell(2\ell n+1)}$ means to add a second coordinate to each of their vertices. We will also add a vertex to the cycles in $\cal B$ so that they become $\ell$-cycles.

This lift is easily done  for the set $\cal A$; from each cycle $A\in \cal A$, $A=(a_0=0,a_1,\dots,a_{4k})$,  we obtain the cycle $A'=(a_0',\dots,a'_{4k})$ by setting
$a'_i=(a_i,i),$ and we set $\cal A'$ to be the set $\{A' | A\in \cal A\}$. It is important to note that
\begin{align}\label{A}
& \text{the projection of} \; \Delta A'_i \; \text{on}\; \Z_{2\ell n+1} \; \text{is}\; \Delta A_i,\\ 
& \text{the projection of} \; \Delta A'_i \; \text{on}\; \Z_{\ell} \; \text{is a list of $\pm1$'s},
 \end{align}
for $i=1,\dots, n-2k$.

We will now work with the set $\mathcal{B}=\{B_1,\ldots,B_{2k}\}$, where
$B_i=(b_{0,i},$ $b_{1,i}, \ldots,$ $b_{4k-1,i})$ with $b_{0,i}=0$ for $i=1, \ldots, 2k$.  
When dealing with the set $\mathcal{B}$ we need to add an extra vertex to each $B_i$ to have an $\ell$-cycle of $K_{2\ell n+1}$ {\em and} a second coordinate to obtain cycles in $K_{\ell(2\ell n+1)}$.
The lift will work differently in the case 
$i=2k$, so we shall start by describing the situation for $i=1,\dots,2k-1$.

Consider the cycles $Q_i=(0,q_{1,i},\dots q_{4k,i})$ constructed in Lemma \ref{pathi}, and consider the $\ell$-cycle
$B'_i= ( (0, 0), (b_{1,i},q_{1,i}), \dots, (b_{4k-1,i},q_{4k-1,i}), (0,q_{4k,i}))$.
By following this construction to lift $B_i$ we have that 
\begin{align}
  &\text{the projection of} \; \Delta B'_i \; \text{on}\; \Z_{2\ell n+1} \; \text{is}\; 
  \Delta B_i\ \cup   \ \{0,0\}, \label{condiz1}\\
  &\text{the projection of} \; \Delta B'_i \; \text{on}\; \Z_{\ell} \; \text{is}\; 
  \Delta Q_i. \label{condiz2}
\end{align}
By recalling that $q_{4k,i}=-i$ and 
$\Delta Q_i\setminus\{\pm q_{4k,i}\}$ contains only $\pm1$s and $\pm2$s (Lemma \ref{pathi}), 
we then have that for any $i=1,\dots, 2k-1$,
\begin{align}
  &\text{the differences of}\; B'_i \;\text{in}\; \{0\} \times \Z_\ell 
  \;\text{are}\;  \pm(0,q_{4k,i}) = \pm(0,i) \label{condiz3}\\
  &\text{the differences of}\; B'_i \;\text{in}\; \Z_{2\ell n+1}^*\times\Z_{\ell} \; 
  \text{lie in} \; \Delta B_i \times \{\pm1,\pm2\}. \label{condiz4}
\end{align}

A similar, but slightly more complicated approach will be used when lifting $B_{2k}$; 
to obtain the result, we need $\ell$ possible lifts for 
$B_{2k} = (0,b_{1,2k},$ $\dots,b_{4k-1,2k})$, 
say  $B'_{2k,m}$ for $m \in \Z_\ell$, 
corresponding to the fact that we want condition (\ref{allvalues})
to be satisfied for all $s\in\Z_\ell$. 
Once more, we add a vertex to $B_{2k}$, and a second coordinate using now the cycles $P_m=(0,p_{1,m},\dots p_{4k,m})$ built in Lemma \ref{pathx}.
For $m \in \Z_\ell$, $m\ne \pm 2k$, set 
$$B'_{2k,m}=( (0, 0), (b_{1,2k},p_{1,m}), \dots, (b_{4k-1,2k},p_{4k-1,m}), (0,p_{4k,m}));$$
on the other hand for $m = \pm 2k$, the difference with $0$ as first coordinate will appear from positions $4k-1$ and $4k$; in this case set 
$$B'_{2k,m}=( (0, 0), (b_{1,2k},p_{1,m}), \dots, (b_{4k-1,2k},p_{4k-1,m}), (b_{4k-1,2k},p_{4k,m})).$$

We first note that $B'_{2k,m}$ satisfies conditions equivalent to \eqref{condiz1}-\eqref{condiz2}. 
Also, by Lemma \ref{pathx}  we have that 
$p_{4k,m}=2k$ for $m\neq \pm 2k$, $p_{4k,m}-p_{4k-1,m}=m$ for $m=\pm 2k$ and
$\Delta P_m \setminus\{\pm 2k\}$ contains only $\pm1$s and $\pm2$s. 
It then follows by the structure of $B'_{2k,m}$ that for any $m\in \Z_{\ell}$
\begin{align}
  &\text{the differences of}\; B'_{2k,m} \;\text{in}\; \{0\} \times \Z_\ell 
  \;\text{are}\;  \pm(0,2k) \label{condiz5}\\
  &\text{the differences of}\; B'_{2k,m} \;\text{in}\; \Z_{2\ell n+1}^*\times\Z_{\ell} \; 
  \text{lie in} \; \Delta B_{2k,m} \times \{\pm1,\pm2\}. \label{condiz6}
\end{align}

In this way, taking into account \eqref{condiz3}-\eqref{condiz6}
any set $\mathcal{S}$ of possible lifts of the $B_i$'s that we obtain
will satisfy the first part of the assertion, that is,
\begin{equation}\label{C}
\begin{aligned}
    \Delta \mathcal{S} = &(\{0\}\times \Z_{\ell}^*) \ \cup \ \{(x, \varphi_{\mathcal{S}}(x)) \;|\; x \in Z_{2\ell n+1}^-\setminus (D\ \cup \ -D)\}\\
    &\text{where}\; \varphi_{\mathcal{S}}: \Z_{2\ell n+1}^-\setminus (D \ \cup \ -D) \rightarrow \{\pm 1, \pm 2\}.
\end{aligned}
\end{equation}
Finally, consider the set ${\cal S}_{m}=\{A'_1,\dots,A'_{n-2k},B'_1,\dots,B'_{2k-1}, B'_{2k,m}\}$
consisting of the lifts of the cycles in $\mathcal{A}$ and $\mathcal{B}$ just described. In view of the above considerations, for $m\in \Z_{\ell}$ the set 
$\mathcal{S}_m$ satisfies \eqref{C}.
For the sake of readability, we set $\varphi_{m} = \varphi_{\mathcal{S}_{m}}$ and 
${\sum}_m = \sum_{i\in \ol{D}} (-1)^i\varphi_{m}(i)$ where $\ol{D}=[2, \ell n-1]\setminus D$.
We want to evaluate the contribution of the cycle $B'_{2k,m}$ to the quantity  
$\sum_m$; to do so, note that by Lemma \ref{skolem} we have
 $\Delta B_{2k}=\pm[u,u']$ where $u=2(2k-1)n$ and $u'=u+4k-1$, and set
\[\displaystyle S_{2k,m}=\sum_{i=u}^{u'} (-1)^i\varphi_{m}(i) \;\;\;\text{for any $m \in \Z_{\ell}$}.\] 
On the other hand, the contribution given by ${\cal S}_m \setminus \{B'_{2k,m}\}$ to $\sum_m$ does not depend on $m$, so we can write ${\sum}_{m} = S_{2k,m} + S'$ for some fixed $S'$.
Since $B_{2k}$ is alternating (Lemma \ref{skolem}), we have that
$B'_{2k,m}$ satisfies the assumption of Lemma \ref{specialcycles} 
for any $m \in \Z_\ell$. By Lemma \ref{pathx} we have that 
$p_{4k} - 2p_{2k}=m$ for $m\neq \pm 2k$; 
also, $p_{4k-1} - p_{4k} - 2p_{2k}=m$ for $m = \pm 2k$. It then
follows by Lemma \ref{specialcycles} that $S_{2k,m}=m$ for any $m \in \Z_\ell$
hence, ${\sum}_{m} = m + S'$.
As $m$ runs over $\Z_\ell$, we have that $\sum_{m}$ covers all integers in $\Z_\ell$ hence, 
condition 1 is proven.

We are left to show that condition \ref{uni} holds. For
$i\in [2, 2k-1]$   
we recall that  $Q_i=(0,q_{1,i}, \ldots, q_{4k,i})$ is the cycle of Lemma \ref{pathi} hence,
$q_{h,i}=h$ for $h\in[1,2k]$. Therefore, the lift $B_i'$ of the cycle $B_i$ has the following form: 
$B'_i= ((0,0),(b_{1,i},1),\ldots, (b_{2k,i}, 2k),\ldots)$. 

Since $B_i$ is a cycle as in Lemma \ref{skolem} then, 
$B_i$ is alternating and for any $h \in [1,4k]$ we have 
that $\delta_{h,i} = (-1)^h (b_{h,i}-b_{h-1,i})\in \ol{D}$.
By Definition \ref{alternating}, $\delta_{h,i} \equiv h+1 \pmod{2}$ and $q_{h,i}-q_{h-1,i}=1$ for  $h \in [1,2k]$, hence 
$\varphi(\delta_{h,i})= (-1)^h(q_{h,i}-q_{h-1,i}) = (-1)^{\delta_{h,i}+1}$ for $h\in [1,2k]$.
In other words, set $X=\{\delta_{h,i} \; | \; (h,i)\in [1,2k]\times[2,2k-1]\}$, we have that
$\varphi(x)=(-1)^{x+1}$ for any $x\in X$, where $X\subseteq \ol{D}$ has size $(2k-2)2k$ and this proves condition 2.
\end{proof}

\begin{ex}\label{exshortcycles}
Once again, let $\ell=9$ and $n=5$, so that  $D=\{2,5\}$.
Here, we explicitly construct a set {\cal S} of
{transversal} short cycles as in Proposition \ref{shortcycles}
with $s=0$.

 First, we consider the sets ${\cal A} = \{A_1\}$ and ${\cal B} = \{B_1, B_2, B_3, B_4\}$
  of short cycles with vertices in $\Z_{91}$ from Example \ref{exskolem}:
\begin{align*}
A_1 &=(-3,25,7,15,24,5,34,-4,36),  \\
B_1 &= (0,-3, 1,  -5, 2, -40, 3,-41), \;\;
B_2 = (0,-10,1, -11, 2, -13, 3,-14), \\ 
B_3 &= (0,-20,1, -21, 2, -23, 3,-24), \;\; 
B_4 = (0,-30,1, -31, 2, -33, 3,-34).
\end{align*}
We lift these cycles to $\Z_{9\cdot91}$  according to Proposition \ref{shortcycles}. 
First  we have that 
$A'_1=((-3,0),(25,1),(7,2),(15,3),(24,4),(5,5),(34,6),(-4,7),(36,8)).$\\
Let us ``expand and lift'' the first three cycles of $\cal B$. We will need the three cycles from Lemma \ref{pathi} which in our case are $Q_1=(0,1,2,3,4,5,6,7,8)$, $Q_2=(0,1,2,3,4,5,6,8,7)$ and $Q_3=(0,1,2,3,4,5,7,8,6)$, so that the new cycles are
\begin{align*}
B'_1 = & ((0,0), (-3,1),(1,2), (-5,3),(2,4),(-40,5),(3,6),(-41,7),(0,8)),\\
B'_2 = & ((0,0),(-10,1),(1,2),(-11,3),(2,4),(-13,5),(3,6),(-14,8),(0,7)), \\
B'_3 = & ((0,0),(-20,1),(1,2),(-21,3),(2,4),(-23,5),(3,7),(-24,8),(0,6)). 
\end{align*}
For the  cycle $B_4$ we need 9 different lifts $B'_{4,m}$ obtained through the cycles $P_m$, $m=0,1,\ldots,8$, of Lemma \ref{pathx}:
\begin{align*}
P_0 = & (0,7,8,1,2,3,5,6,4) & P_1= & (0,1,8,7,6,5,3,2,4) \\
P_2 = & (0,7,6,8,1,2,3,5,4) & P_3= & (0,8,7,6,5,3,1,2,4) \\
P_4 = & (0,2,3,4,5,7,8,6,1) & P_5= & (0,7,6,5,4,2,1,3,8) \\
P_6 = & (0,2,3,1,8,7,6,5,4) & P_7= & (0,8,1,2,3,5,7,6,4) \\
P_8 = &(0,2,1,8,7,6,5,3,4).
\end{align*}
For example, when $m=6$, the lift $B'_{4,6}$ of $B'_4$ via $P_6$ is: 
\[
B'_{4,6}= ((0, 0),(-30, 2),(1, 3), (-31, 1), (2, 8), (-33, 7), (3, 6), (-34,5),(0, 4)). 
\]
Finally, we set ${\cal S}_m = \{A'_1, B'_1, B'_2, B'_3, B'_{4,m}\}$ 
for any $m \in \Z_9$. Of course,
$\Delta {\cal S}_m = \Delta A'_1 \ \cup \ $ $\Delta B'_1 \ \cup \ \ldots \ \cup \
\Delta B'_{4,m}$ where:
\begin{equation}\label{deltas}
\begin{aligned}
  & \Delta A'_1 = \pm(\{8,9,28,29,40\}\times\{1\}) \ \cup \ 
  \pm(\{18,19,38,39\}\times\{-1\}), \\
  & \Delta B'_1 = \pm(\{4,7,41,43\}\times\{1\}) \ \cup \ 
  \pm(\{3,6,42,44\}\times\{-1\}) \ \cup \ \{\pm(0,1)\}, \\ 
  & \Delta B'_2 = 
  \pm(\{11,13,16\}\times\{1\}) \ \cup \ \pm(\{10,12,14,15\}\times\{-1\}) \\
  & \quad\quad\quad\; \cup \ \{\pm(0,2), \pm(17,-2)\}, \\ 
  & \Delta B'_3 = 
  \pm(\{21,23\}\times\{1\}) \cup \ \pm(\{20,22,25,27\}\times\{-1\}) \\
  &  \quad\quad\quad\; \cup \ \{\pm(0,3), \pm(24,-2), \pm(26,2)\}.
\end{aligned}
\end{equation}
In view of Lemma \ref{pathx}, the reader can check that
$\Delta B'_{4,m} \subset \Delta B_4 \times \{\pm1, \pm2\} \ \cup \ \{(0,\pm4)\}$
for any $m\in \Z_9$. For example, when $m=6$ we have 
\begin{align*}     
  & \Delta B'_{4,6} = 
  \pm(\{31,35,37\}\times\{1\}) \cup \ \pm(\{34,36\}\times\{-1\}) 
  \cup \ \pm(\{30,33\}\times\{-2\})\\
  &  \quad\quad\quad\; \cup \ \{\pm(0,4), \pm(32,2)\}.     
\end{align*}
Since the projection of $\Delta {\cal S}_m$ on $\Z^*_{91}$ is the 
set $\Delta A_1 \ \cup \ $ $\Delta B_1 \ \cup \ \ldots \ \cup \
\Delta B_{4}$ which therefore does not have repeated elements, we are guaranteed 
that there exists a map $\varphi_{m}: 
\Z_{91}^-\setminus (D \ \cup \ -D) \rightarrow \{\pm 1, \pm 2\}$, 
where $D=\{2, 5\}$, that allows us to describe $\Delta {\cal S}_m$ as follows:
\begin{align*}
   \Delta \mathcal{S}_m = &(\{0\}\times \Z_{9}^*) \ \cup \ 
    \{(x, \varphi_{m}(x)) \;|\; x \in Z_{91}^-\setminus (D\ \cup \ -D)\}.
\end{align*}
Note that  $\varphi_{m}(-x) = -\varphi_{m}(x)$ for any 
$x \in Z_{91}^-\setminus (D\ \cup \ -D)$.
In particular, for $m=6$ the map $\varphi_m$ acts on 
$[2,44]\setminus D$  as follows:
\begin{align*}
&\{(x, \varphi_{6}(x))\mid x \in [3,44]\setminus \{5\}\} = 
 \{(3, -1), (4, 1), (6, -1), (7, 1),(8, 1),\\
&\;\;  (9, 1), (10, -1), (11, 1), (12, -1), (13, 1), (14, -1), (15, -1), (16, 1), (17, -2),\\
&\;\;  (18, -1), (19, -1), (20, -1), (21, 1),(22, -1), (23, 1), (24, -2), (25, -1), (26, 2),\\
&\;\;  (27, -1), (28, 1), (29, 1), (30, -2),  (31, 1), (32, 2), (33, 7), (34, -1), (35, 1),\\
&\;\;  (36, -1), (37, 1), (38, -1), (39, 8),(40, 1), (41, 1), (42, -1), (43, 1), (44, -1)\}.
\end{align*} 
By \eqref{deltas}, we can easily see that condition 2 of Proposition \ref{shortcycles} is satisfied.
We are then left to compute the sum ${\sum}_m= \sum_{i\in \ol{D}} (-1)^i\varphi_{m}(i)$ where $\ol{D}=[2,44]\setminus D$. Note that the projection of 
$\Delta {B'_{4,m}}$ on $\ol{D}$ is  $[30,37]$.
Therefore, ${\sum}_{m} = S' + S_{4,m}$ where $S'$ and $S_{4,m}$ are the contributions 
given by ${\cal S}_m \setminus \{B'_{4,m}\}$ and $B'_{4,m}$ to $\sum_m$, respectively. 
In other words, for any $m\in\Z_9$ we have
\[
  S'= \sum_{i\in\ol{D} \setminus [30,37]} (-1)^i\varphi_{m}(i) = 3 
  \quad\text{and}\quad
  S_{4,m}=\sum_{i=30}^{37} (-1)^i\varphi_{m}(i). 
\]
Note that $S'$ does not depend on $m$.
According to Lemma \ref{pathx} we have that $S_{4,m}=m$: the reader can easily check that 
$S_{4,6}=6$ (check also Example \ref{exspecialcycles} for the case $m=4$).
Therefore, ${\sum}_{m} = m + S'$ for any $m \in \Z_9$; in particular, ${\sum}_{6} = 0$.
\end{ex}

\section{The main result}\label{main}
In this section we finish proving our main result, Theorem \ref{mainth}. We are trying to build a set of base cycles such that its list of differences is $\Z_{\ell\cdot(2\ell n+1)} \setminus \{0\}$; this will be done by completing the set 
${\cal L} \ \cup \ {\cal S}$ of {transversal} cycles we have built in the two previous sections with two more {transversal} long cycles $C$ and $C'$ providing the $4\cdot(2\ell n+1)$ missing differences. These differences are a particular subset of elements of the form  $\Z_{2\ell n+1}^*\times\{\pm 1,\pm 2\}$, together with the elements  of the form $\Z_{2\ell n+1}^*\times\{0\}$.
 
Much of the work in this section will be to ensure that   the differences from 
$\Z_{2\ell n+1}^*\times\{\pm 1,\pm 2\}$ that appear in $C$ and $C'$ 
have not been already covered in ${\cal L} \ \cup \ {\cal S}$: 
to this end, first we build an auxiliary function $G$.

Given a map $F:\Z_{2\ell n+1}^- \rightarrow \{\pm 1, \pm2\}$, we will briefly denote with $\sum F$ the integer
$\displaystyle\sum_{i=2}^{\ell n-1} (-1)^i F(i)$.
Also, recall that 
$\ol{D}=[2, \ell n-1]\setminus D$.

\begin{lem}\label{G} Let $F:\Z_{2\ell n+1}^- \rightarrow \{\pm 1, \pm2\}$ be a map such that 
\[
  \sum F \equiv 0 \pmod{\ell}, \quad \quad F(-x)=-F(x) \;\; \text{for every} \; x\in \Z_{2\ell n+1}^-,
\] 
and {assume that} the set $\{x \in \ol{D} \;|\; F(x)=(-1)^{x+1}\}$ has size $\geq \ell-1$.
Then, for every integer $\rho$ there exists a map $G:\Z_{2\ell n+1}^- \rightarrow \{\pm 1, \pm2\}$ such that:
\begin{enumerate}
\item $\sum G\equiv \rho\pmod{\ell}$;
\item $|G(x)| = 1$ (or $2$) if and only if $|F(x)| = 2$ (or $1$);
\item $G(-x)=-G(x)$ for every $x\in \Z_{2\ell n+1}^-.$  
\end{enumerate}
\end{lem}
\begin{proof}
Let $g:\Z_{2\ell n+1}^- \rightarrow \{\pm 1, \pm2\}$ be the map defined as follows:
\[
  g(x)=
  \begin{cases}
    2F(x) & \text{if $F(x)=\pm 1$},\\
    \frac{F(x)}{2} & \text{if $F(x)=\pm 2$}.
  \end{cases}
\]
We will get the map required in the statement by slightly modifying $g$. \\
Let $t\in [0,\ell-1]$ such that $t \equiv (\sum g - \rho)(\ell-1)/4 \pmod{\ell}$.
By assumption, there is a set $X \subseteq \ol{D}$ of size $t$ such that $F(x)=(-1)^{x+1}$ for any $x \in X$.
We define a map $G:\Z_{2\ell n+1}^- \rightarrow \{\pm 1, \pm2\}$  as follows
\[
G(x)=
  \begin{cases}
    -g(x) & \text{if $x\in X \ \cup \ -X$},\\
    g(x) & \text{otherwise}.
  \end{cases}
\]
Note that $g(x) = (-1)^{x+1} \cdot 2$ for $x\in X$, hence
\begin{equation}\label{g+4}
  -g(x) = g(x) + (-1)^x\cdot 4 \;\;\;\text{for any}\;\;\; x\in X.
\end{equation}

It is easily seen, keeping in mind the definition of $g$ and the fact that $|G(x)|=|g(x)|$ for $x \in \Z_{2\ell n+1}^-$, that
Property 2. and 3. hold. Let us prove that also the first property is satisfied.

Let $S_1, S_2,$ and $S_3$ be the partial sums defined below:
\[S_1= \displaystyle\sum_{i\in \overline{X}} (-1)^{i}G(i), \;\;\; S_2=\displaystyle\sum_{i\in X_e} G(i), \;\;\; \mbox{and} \;\;\; S_3=-\displaystyle\sum_{i\in X_o} G(i).\]
where $X_e$  (resp. $X_o$) denotes the set of even (resp. odd) integers in $X$, and $\overline{X}=[2,\ell n-1]\setminus X$; of course, 
\[\sum G = S_1 + S_2 + S_3 \;\;\; \text{and}\;\;\; t=|X| = |X_e|+|X_o|.\] 
Taking \eqref{g+4} into account and recalling how $G$ is defined, it is straightforward to check that we have:
\[
  S_1 = \displaystyle\sum_{i\in \overline{X}} (-1)^i g(i), \;\;\;\;\;
  S_2 = \displaystyle\sum_{i\in X_e} -g(i) = \displaystyle\sum_{i\in X_e} (g(i) + 4) =
  \displaystyle\sum_{i\in X_e} g(i) + 4\cdot|X_e|,
\]
\[
  \mbox{and} \;\;\; S_3 = \displaystyle\sum_{i\in X_o} g(i) = \displaystyle\sum_{i\in X_o} (-g(i) + 4) =
  \displaystyle\sum_{i\in X_o} -g(i) + 4\cdot|X_o|.
\]
By recalling how $t$ is defined, it follows that
\[
\begin{aligned}  
  \sum G = &\sum_{i=2}^{\ell n-1} (-1)^ig(i) + 4\cdot (|X_e| + |X_o|) = \\ 
           &\sum g + 4t \equiv \ell(\sum g - \rho) + \rho \equiv \rho \pmod{\ell}.
\end{aligned}
\]
We have therefore proven that Property 1. holds and this completes the proof.
\end{proof}

We are now able to prove the main result of this paper.\\

\noindent
{\bf Theorem  \ref{mainth}.}\,
{\it
HWP$(\ell(2\ell n+1);[\ell^{2\ell n+1}],[(2\ell n+1)^\ell];\ell n,
{\frac{(\ell-1)(2\ell n+1)}{2}})$
admits a cyclic solution for any $\ell \equiv 1 \pmod{4}$ and $n\geq (\ell-1)/2$.
}
\begin{proof}
Let $\ell=4k+1$ and assume $n\geq 2k$. 
Once more we set $k\geq 2$, since the assertion has been proven in \cite{BuDa} for $k=1$.

We first take a set ${\cal L}$ of $2k-2$ transversal $(2\ell n +1)$-cycles as in Proposition \ref{longcycles},  and set $s=-\sum_{i\in D} (-1)^i \ f(i)$, where $f$ is the map from Proposition \ref{longcycles}.
Then we take a set ${\cal S}$ of $n$ transversal $\ell$--gons as in Proposition \ref{shortcycles}, choosing $s$ as above: 
it then follows that 
$s= \sum_{i\in \ol{D}} (-1)^i \varphi(i)$ where $\ol{D}=[2,\ell n-1]\setminus D$
and $\varphi$ is the map  from  Proposition \ref{shortcycles}. 

To apply Theorem \ref{BR} we have to find two transversal $(2\ell n+1)$-cycles $C$ and $C'$ of $K_{\ell(2\ell n+1)}$  whose differences coincides with the complement of $\Delta{\cal S} \ \cup \ \Delta{\cal L}$ in $(\Z_{2\ell n+1}\times\Z_\ell)\setminus\{(0,0)\}$.

Let $F$ be the map $F:\Z_{2\ell n+1}^- \rightarrow \{\pm 1, \pm2\}$ obtained by glueing together the maps $f$  and  
$\varphi$: 
\begin{equation}\label{mapF}
F(x)=
  \begin{cases}
    f(x) & \text{if $x\in D \ \cup \ -D$},\\
    \varphi(x) & \text{if $x\in \ol{D} \ \cup \ -\ol{D}$}.
  \end{cases}
\end{equation}
Note that we can write $\Delta \mathcal{S} \ \cup \ \Delta \LL$ as the disjoint union of the following sets:
\begin{equation}
   \Delta_1 = \{0\} \ \times \ \Z_{\ell}^*; \quad\quad
   \Delta_2= \Z^*_{2\ell n+1} \ \times \ \Z_{\ell}\setminus\{0,\pm1,\pm2\};
\end{equation}
\[\displaystyle \Delta_3 = \cup_{i\in \Z_{2\ell n+1}^-} \{i\} \ \times \ \{F(i)\}.
\]
One can check that $F$ satisfies the assumptions of Lemma \ref{G}. In fact, 
\begin{equation}\label{sumF}
\sum F  = s + \sum_{i\in D} (-1)^i \ f(i) \equiv 0\pmod{\ell}.
\end{equation}
Moreover, it is clear that the map $F$ has the property that $F(-x)=  F(x)$ for every $x \in \Z_{2\ell n+1}^-$ since $\Delta \mathcal{S} \ \cup \ \Delta \LL$ is obviously symmetric.
Finally, Proposition \ref{shortcycles}.(2) ensures that $\{x \in \ol{D} \;|\;	F(x)=(-1)^{x+1}\}$
has size $\geq \ell-1$. Then, by Lemma \ref{G} there is a map $G:\Z_{2\ell n+1}^- \rightarrow \{\pm 1, \pm2\}$ such that
\begin{equation}\label{sumG}
  \sum G \equiv -1\pmod{\ell};
\end{equation}
\begin{equation}\label{GandF}
  |G(x)| = 1 \;\; (\text{or}\; 2) \quad \text{if and only if} \quad |F(x)| = 2 \;\;(\text{or}\; 1);
\end{equation}
\begin{equation}\label{oddmap}
  G(-x)=-G(x)\quad \text{for every}\quad x\in \Z_{2\ell n+1}^-.
\end{equation}
Let us construct the {transversal} 
$(2\ell n+1)$-cycle $C=(c_0,c_1,\dots,c_{2\ell n})$ of $K_{\ell(2\ell n+1)}$ with $c_i=(x_i,y_i) \in \Z_{2\ell n+1}\times\Z_\ell$ defined as follows:

\[x_i=(-1)^{i+1}\biggl{\lfloor}{\frac{i+1}{2}}\biggl{\rfloor}\quad\mbox{for }0\leq i\leq 2\ell n;\]

\[
y_i=
\begin{cases}
  1+\displaystyle\sum_{j=2}^i (-1)^j F(j) & \text{for $i\in [2,\ell n-1]$,}\\
  1+\displaystyle\sum_{j=\ell n+2}^i (-1)^j G(j) & \text{for $i \in [\ell n+2, 2\ell n-1]$;}
\end{cases}
\]

$$(y_0,y_1,y_{\ell n},y_{\ell n+1},y_{2\ell n})=(0,1,2,1,-2).$$

Now construct another {transversal} $(2\ell n+1)$-cycle $C'=(c'_0,c'_1,\dots,c'_{2\ell n})$
of $K_{\ell(2\ell n+1)}$ with $c'_i=(x'_i,y'_i)$ defined as follows.

\medskip
1st case: $n$ is even.
\[
x'_i=
\begin{cases}
 x_i &\mbox{for $i \in [0, \ell n]$,} \\
-x_i &\mbox{for $i \in [\ell n+1, 2\ell n]$;}
\end{cases}\quad\quad
y'_i=
\begin{cases}
 0 & \mbox{for $i\in [0,\ell n]$,} \\
y_i &\mbox{for $i \in [\ell n+1, 2\ell n]$.}
\end{cases}
\]

\medskip
2nd case: $n$ is odd.
\[
x'_i=
\begin{cases}
 x_i &\mbox{for $i \in [0, \ell n-1]$,} \\
-x_i &\mbox{for $i \in [\ell n, 2\ell n]$;}
\end{cases}\quad\quad
y'_i=
\begin{cases}
 0 & \mbox{for $i\in [0,\ell n-1]$,} \\
 1 & \mbox{for $i=\ell n$,} \\
y_i &\mbox{for $i \in [\ell n+1, 2\ell n]$.}
\end{cases}
\]

Let $\pi(C)$ and $\pi(C')$ be the projections of $C$ and $C'$ on $\Z_{2\ell n+1}$. It is clear that
\[\pi(C)=(0,1,-1,2,-2,\dots,\ell n,-\ell n).\]
The expression of $\pi(C')$ is similar but there is a twist in the middle, namely,
\[ \pi(C')=(0,1,-1,2,-2,\dots,\nu,-\nu,-(\nu+1),\nu+1,\dots,-\ell n,\ell n), \;
\text{and $\nu = \left\lfloor\frac{\ell n}{2}\right\rfloor$}.
\]
In any case the above remark guarantees that both $C$ and $C'$ are transversals. Now we need to
calculate the lists of differences of $C$ and $C'$. This is easily done, but  it is important to note first that we have
\begin{equation}\label{f implies}
y_{\ell n-1}=1 \quad\quad \text{and} \quad\quad y_{2\ell n-1}=0.
\end{equation}
Indeed, by definition, we have $y_{\ell n-1}=1+\sum F$.
Again by definition, we have $y_{2\ell n-1}=1+\sigma$ with
$\sigma=\sum_{i=\ell n+2}^{2\ell n-1}(-1)^iG(i)$.
Taking \eqref{oddmap} into account, it is straightforward to check that 
$\sigma = \sum_{i=2}^{\ell n-1}(-1)^{i+1}G(-i) = \sum G$. Hence, in view of \eqref{sumF} and \eqref{sumG}, we have that (\ref{f implies}) holds.

Let us consider the following subpaths of the cycle $C$:
\[
P_1=(c_1,c_2,\dots,c_{\ell n-1});\quad
P_2=(c_{\ell n-1},c_{\ell n},c_{\ell n+1});
\]
\[
P_3=(c_{\ell n+1},c_{\ell n+2},\dots,c_{2\ell n-1});
\quad P_4=(c_{2\ell n-1},c_{2\ell n},c_{0},c_1).
\]
Taking (\ref{f implies}) into account, it is straightforward to check that we have:
\begin{description}
\item $\Delta P_1=\pm\{(i,-F(i)) \ | \ 2\leq i\leq \ell n-1\}$;
\item $\Delta P_2= \pm\{(\ell n,-1),(\ell n+1,-1)\} = 
\pm\{(\ell n,-1),(\ell n,1)\}$;
\item $\Delta P_3 = \pm\{(i,-G(i)) \ | \ \ell n+2\leq i\leq 2\ell n-1\}
 = \pm\{(i,-G(i)) \ | \ 2\leq i\leq \ell n-1\}$;
\item $\Delta P_4=\pm\{(1,-2),(\ell n,2),(1,1)\}$.
\end{description}
Now consider the following subpaths of the cycle $C'$:
\[
P'_1=(c'_0,c'_1,\dots,c'_{\ell n});\quad
P'_2=(c'_{\ell n},c'_{\ell n+1});
\]
\[
P'_3=(c'_{\ell n+1},c'_{\ell n+2},\dots,c'_{2\ell n-1});\quad P'_4=(c'_{2\ell n-1},c'_{2\ell n},c'_{0}).
\]
Also here, in view of (\ref{f implies}), we can easily check that we have:
\begin{description}
\item
$\Delta P_1'=
\begin{cases}
\Z_{2\ell n+1}^*\times\{0\} & \text{for $n$ even;}\\
(\Z_{2\ell n+1}^*\setminus\{\pm\ell n\})\times\{0\} \ \cup \ \pm\{(1,-1)\} & \text{for $n$ odd;}
\end{cases}
$
\item
$
\Delta P_2'=
\begin{cases}
\pm\{(1,-1)\} & \mbox{for $n$ even;}\\
\pm\{(\ell n,0)\}& \mbox{for $n$ odd;}
\end{cases}
$
\item $\Delta P_3'=\pm\{(i,G(i)) \ | \ 2\leq i\leq \ell n-1\}$; \quad $\Delta P_4'=\pm\{(1,2),(\ell n,-2)\}$.
\end{description}

It is evident that $C$ is the union of the pairwise edge-disjoint paths $P_i$s and that $C'$ is the union of the  edge-disjoint paths  $P'_i$s;  we can therefore write:
$$\Delta\{C,C'\}=\Delta\{P_1,P_2,P_3,P_4,P'_1,P'_2,P'_3,P'_4\}.$$
In this way we see that $\Delta\{C,C'\}$ is the union of the following pairwise disjoint lists:
\[\Delta_4=\Z_{2\ell n+1}^*\times\{0\}; \quad\quad
\Delta_5=\{1,\ell n,\ell n+1,2\ell n\}\times \{\pm1,\pm2\};
\]
\[
\Delta_6=\bigcup_{i\in\Z_{2\ell n+1}^-} \{i\}\times\{-F(i),G(i),-G(i)\}.
\]
Note that $\{1,\ell n,\ell n+1,2\ell n\}=\Z_{2\ell n+1}^*\setminus \Z_{2\ell n+1}^-$.
Also, taking \eqref{GandF} into account we can write $\{-F(i),G(i),-G(i)\}=\{\pm1,\pm2\}\setminus\{F(i)\}$ for every $i\in\Z_{2\ell n+1}^-$.
Therefore, recalling that $\Delta_3 = \bigcup_{i\in \Z_{2\ell n+1}^-} \{i\} \ \times \ \{F(i)\}$, we have that
\[\Delta_3 \ \cup \ \Delta_5 \ \cup \ \Delta_6 = \Z^*_{2\ell n+1} \ \times \ \{\pm1, \pm2\}. \]
We then conclude that $\Delta\{C,C'\} \ \cup \ \Delta{\cal S} \ \cup \ \Delta{\cal L} = \Z_{2\ell n+1}\times\Z_{\ell}\setminus\{(0,0)\}$, namely, $\Delta\{C,C'\}$  is exactly the complement of 
$\Delta{\cal S} \ \cup \ \Delta{\cal L}$ in $\Z_{2\ell n+1}\times\Z_{\ell}\setminus\{(0,0)\}$.
This means that ${\cal S}$ and ${\cal L}'={\cal L}\cup\{C,C'\}$ are two sets of short and long transversal cycles, respectively, as required by Theorem \ref{BR},
and the assertion follows.
 \end{proof}

 \begin{ex}  
 Here, we solve HWP$(9\cdot 91;[9^{91}],[91^9];45, 364)$ 
  by showing the existence of two sets
of short and long transversal cycles as required by Theorem \ref{BR}. 

 Let $\ell=9$ and $n=5$. Note that $D=\{2,5\}$ and $\ol{D}=[2,44]\setminus\{2,5\}$; 
 also, $\Z^-_{91} =\pm [2,44]$. 
 The set ${\cal L}=\{C_1',C_2'\}$ built in Example \ref{exlongcycle} according to Proposition \ref{longcycles} is such that
\begin{equation}\label{fD}
 \Delta {\cal L} = \Z_{91}^* \times \Z_9 \setminus \{0,\pm1,\pm2\} \cup \
             \{\pm(2,f(2))), \pm (5,f(5))\}.
\end{equation}
where 
$f(2) = f(5) = -1$.
Now set $s=-\sum_{i\in D} (-1)^i \ f(i) = 0$. We need a set ${\cal S}$ of $5$ transversal $9$--gons as in Proposition \ref{shortcycles}, choosing $s$ as above. It is enough to take ${\cal S}$ to be the set of 
{transversal} short cycles constructed in Example \ref{exshortcycles} for $m=6$.

Let $F:\Z_{2\ell n+1}^- \rightarrow \{\pm 1, \pm2\}$ be the map \eqref{mapF} 
obtained by glueing 
together the maps $f$  and  $\varphi$. Recall that $F(-i)=-F(i)$ for any $i\in \Z^-_{91}$; also, 
by collecting the values of $\varphi$ from Example \ref{exshortcycles}, we have that
\begin{footnotesize}
\begin{align*}
\{ (i,   F(i)) \;|\; i \in [2,44] \} = 
  \{(2,8), (3, 8), (4, 1), (5,8), (6, 8), (7, 1), (8, 1), (9, 1) &,  \\
   (10, 8), (11, 1), (12, 8), (13, 1), (14, 8), (15, 8), (16, 1), (17, 7), (18, 8)   &,\\
   (19, 8), (20, 8), (21, 1), (22, 8), (23, 1), (24, 7), (25, 8), (26, 2), (27, 8)  &,\\
   (28, 1), (29, 1), (30, 7), (31, 1), (32, 2), (33, 7), (34, 8), (35, 1), (36, 8)  &,\\
   (37, 1),  (38, 8), (39, 8), (40, 1),(41, 1), (42, 8), (43, 1), (44, 8) & \}.
\end{align*}
\end{footnotesize}
Note that $\Delta \mathcal{S} \ \cup \ \Delta \LL$ is the disjoint union of the following sets:
\[
   \Delta_1 = \{0\}  \times  \Z_{9}^*; \quad
\Delta_2= \Z^*_{91}  \times  \Z_{\ell}\setminus\{0,\pm1,\pm2\}; \quad
\Delta_3 =  \{(i,F(i) \;|\;i\in\Z_{91}^-\}.
\]
We are going to build two remaining transversal 91-cycles $C$ and $C'$ of
$K_{9\cdot91}$  whose differences coincides with the complement of 
$\Delta{\cal S} \ \cup \ \Delta{\cal L}$ in $(\Z_{91}\times\Z_9)\setminus\{(0,0)\}$, that is,
$\Delta C \ \cup \ \Delta C' = \Delta_4 \ \cup \ \Delta_5 \ \cup \ \Delta_6$ where:
\[\Delta_4=\Z_{91}^*\times\{0\}; \quad
\Delta_5=\{\pm1,\pm45\}\times \{\pm1,\pm2\};
\]
\[
\Delta_6=\bigcup_{i\in\Z_{91}^-} \{i\}\times(\{\pm1,\pm2\}\setminus\{F(i)\}).
\]
As a result, we obtain two sets ${\cal S}$ and ${\cal L}'={\cal L}\cup\{C,C'\}$ 
of short and long transversal cycles, respectively, as required by Theorem \ref{BR}. 
This guarantees the existence of a solution to 
HWP$(9\cdot 91;[9^{91}],[91^9];45, 364)$.

Let us start with $C=((x_0,y_0),(x_1,y_1),\dots,(x_{90},y_{90}))$. By construction,
$(x_0, x_1, \ldots,x_{44}) = (0,1,-1,2,-2,\ldots, 45,-45)$; Also,
it is possible to check that the sequence 
$(y_2,y_3,\dots,y_{44})$ of second components is the string
\begin{align*}
(0, 1, 2, 3, 2, 1, 2, 1, 0, 8, 7, 6, 5, 6, 7, 0, 8, 0, 8, 7, 6, 5, 3, 4, 6, 7, 8, 7, \\
5, 4, 6, 8, 7, 6, 5, 4, 3, 4, 5, 4, 3, 2, 1)
\end{align*}
with $y_{44}=1$ as required in (\ref{f implies}). 
We have that $(y_0,y_1,y_{45},y_{46},y_{90})$ are defined in the construction to be $(0,1,2,1,7)$. 

It remains to apply Lemma \ref{G}, with $\rho=-1$, to find the function $G$ and thus 
$(y_{47},\dots,y_{89})$.
The quantity $t$ required in the proof of the lemma is $t \equiv 2(\sum g + 1)\;\pmod{9}$, 
where $g:\Z^-_{91} \rightarrow \{\pm1, \pm2\}$ is the map defined as follows: $g(x) = 2\cdot F(x)$ or $F(x)/2$ according to whether $F(x)=\pm1$ or $F(x)= \pm2$. One can check that $t = 8$. Now, a set $X\subset[2,44]$ of size $t$
with the property that $F(x) = (-1)^{x+1}$ can be $X= \{10,11,12,13,14,18,20,21\}$. We can now obtain 
the required map $G:\Z^-_{91} \rightarrow \{\pm1, \pm2\}$ defined as follows: $G(x)=-g(x)$ or $g(x)$ according to whether $x \in X \ \cup -X$ or not. One can check that $G(-x) = -G(x)$ for $x \in Z^-_{91}$ and
\begin{footnotesize}
\begin{align*}
\{ (i, G(i)) \;|\; i \in [2,44] \} = 
   \{(47,2),(48,7),(49,2),(50,7),(51,7),(52,2),(53,2),(54,7) &,\\
     (55,2),(56,7),(57,2),(58,1),(59,8),(60,7),(61,1),(62,7),(63,7) &,\\
     (64,2),(65,8),(66,2),(67,1),(68,7),(69,2),(70,2),(71,7),(72,2) &,\\
     (73,7),(74,1),(75,7),(76,2),(77,7),(78,2),(79,7),(80,2),(81,7) &,\\
     (82,7),(83,7),(84,7),(85,2),(86,2),(87,7),(88,2),(89,2) &\}.
\end{align*}
\end{footnotesize}The remaining list of second components $(y_{47},\dots,y_{89})$ can then be computed, giving us the list
\begin{align*}
 (8,6,4,2,4,6,4,2,0,7,5,6,7,5,4,2,4,6,7,0,8,6,4,6,8,\\
 1,3,4,6,8,1,3,5,7,0,7,0,7,5,7,0,2,0)
\end{align*}
so that $y_{89}=0$ as needed in (\ref{f implies}). Some minor tweaks are required for the cycle 
$C'=((x'_0,y'_0),(x'_1,y'_1),\dots,(x'_{90},y'_{90}))$; 
here the first components $(x'_0,x'_1,\dots,x'_{90})$ are $(0,1-1,2,-2,\dots,22,-22,$ $-23,23,-24,24, \ldots, -45,45)$, and the second components are simply 
$y'_0=y'_1=\dots=y'_{44}=0$, 
$y'_{45}=1$, and 
$y'_{46}=y_{46}, y'_{47}=y_{47}, \dots, y'_{90}=y_{90}$. We can finally check that:
\[\Delta C = \{(i,-F(i)), (i,-G(i)) \;|\; i\in\Z^-_{91}\} \ \cup \ 
\pm\{(45,\pm1), (1,-2), (45,2), (1,1)\}, \] 
\[\Delta C' = \Z^*_{91}\times\{0\} \ \cup \ \{(i,G(i)) \;|\; i\in\Z^-_{91}\} \ \cup \ 
\pm\{(1,-1), (1,2), (45,-2)\}.\]
Therefore $\Delta C \ \cup \ \Delta C' = \Delta_4 \ \cup \ \Delta_5 \ \cup \ \Delta_6$, 
since $\pm\{F(i), G(i)\} = \pm\{1, 2\}$ for any $i\in \Z^-_{91}$.
 \end{ex}
 
\section*{Acknowledgments}
The authors are very grateful to the anonymous referees for the careful reading,
and for the many comments and suggestions that improved the readability of this paper.

\end{document}